\documentclass[12pt]{article}

\usepackage{amsmath}   % From the American Mathematical Society
\usepackage{amsthm}
\usepackage{mathptmx}

\newtheorem{theorem}{Theorem}
\newtheorem{lemma}{Lemma}
\newtheorem*{lemma*}{Lemma}
\newtheorem{corollary}{Corollary}
\newtheorem{remark}{Remark}

\usepackage[nooneline]{caption}
\usepackage{pifont,hyperref,amsthm,amsmath,amssymb,color,multirow,graphicx,subfigure,color}
\numberwithin{equation}{section}

\RequirePackage{palatino}

\begin{document}

%%% Title section

\begin{center}
{\Large\bf Renewed Limit Theorems for the discrete-time Branching Process and
                its Conditioned \\ Limiting Law interpretation}
\end{center}
\vspace{.1cm}
\begin{center}
{\sc Azam A.~Imomov}\\
\vspace{.4cm}
{\small \it Karshi State University, 17 Kuchabag street, \\
100180 Karshi city, Uzbekistan\\}
{\small e-mail: {\small \sl imomov{\_}\,azam@mail.ru}\\
\vspace{.4cm}
{\textit{\small \textbf{Dedicated to my Parents}}}}
\end{center}

\vspace{.3cm}
%%% abstract
\begin{abstract}
    Our principal aim is to observe the Markov discrete-time process of population growth with long-living
    trajectory. First we study asymptotical decay of generating function of Galton-Watson process for all
    cases as the Basic Lemma. Afterwards we get a Differential analogue of the Basic Lemma. This Lemma plays
    main role in our discussions throughout the paper. Hereupon we improve and supplement classical results
    concerning Galton-Watson process. Further we investigate properties of the population process so called
    Q-process. In particular we obtain a joint limit law of Q-process and its total state. And also we
    prove the analogue of Law of large numbers and the Central limit theorem for total state of Q-process.


\end{abstract}

%%%

\medskip

\section{Introduction }

    The Galton-Watson branching process (GWP) is a famous classical model for population growth.
    Although this process is well-investigated but it seems to be wholesome to deeper discuss and
    improve some famed facts from classical theory of GWP. In first half part of the paper,
    Sections 2 and 3, we will develop discrete-time analogues of Theorems from the paper of the author {\cite{Imomov14a}}.
    These results we will exploit in subsequent sections to discuss properties of so-called
    Q-process as GWP with infinite-living trajectory.

    Let a random function $Z_n $ denotes the successive population size in the GWP at the
    moment $n \in {\mathbb{N}}_0 $, where; ${\mathbb{N}}_0 = \{ 0\} \cup {\mathbb{N}}$ and
    ${\mathbb{N}} = \left\{ 1,2,\ldots \right\}$.  The state sequence
    $\left\{ {Z_n , n \in{\mathbb{N}}_0 } \right\}$ can be expressed in the form of
$$
    Z_{n + 1}  = \xi _{n1}  + \xi _{n2}  + \cdots  + \xi_{nZ_n },
$$
    where $\xi _{nk} $, $n,k \in {\mathbb{N}}_0 $, are independent variables with general offspring
    law $p_k : = \mathbb{P}\left\{{\xi _{11}  = k} \right\}$. They are interpreted as a number of descendants
    of $k$-th individual in $n$-th generation. Owing to our assumption
    $\left\{ {Z_n , n \in {\mathbb{N}}_0 }\right\}$ is a homogeneous Markov chain with state
    space ${\cal S} \subset {\mathbb{N}}_0 $ and transition functions
\begin{equation}
    P_{ij} : = \mathbb{P}\bigl\{ {Z_{n + 1}  = j\bigm|{Z_n  = i} }\bigr\}
    = \sum\limits_{k_1  + \, \cdots \, + k_i  = j} {p_{k_1 }\cdot p_{k_2 } \, \cdots \,p_{k_i } },
\end{equation}
    for any $i,j \in {\cal S}$, where $p_j  = P_{1j}$ and $\sum\nolimits_{j \in {\cal S}} {p_j }= 1$.
    And on the contrary, any chain satisfying to property (1.1) represents GWP with the evolution
    law $\left\{ {p_k ,k \in {\cal S}} \right\}$. Thus, our GWP is completely defined by setting
    the distribution $\left\{ {p_k }\right\}$; see {\cite[pp.1--2]{ANey}}, {\cite[p.19]{Jagers75}}.
    From now on we will assume that $p_k  \ne 1$ and $p_0 > 0$, $p_0  + p_1  < 1$.

    A probability generating function (GF) and its iterations is important analytical tool
    in researching of properties of GWP. Let
$$
    F(s) = \sum\limits_{k \in {\cal S}} {p_k s^k }, \quad \parbox{2.4cm}{\textit{for} {} $0 \le s < 1$.}
$$
    Obviously that $A: = \mathbb{E}\xi_{11} = F'(s \uparrow 1) $ denotes the mean per
    capita number of offspring provided the series $\sum\nolimits_{k \in {\cal S}} {kp_k } $ is finite.
    Owing to homogeneous Markovian nature transition functions
$$
    P_{ij} (n): =\mathbb{P}_i \bigl\{ {Z_n  = j} \bigr\}
    = \mathbb{P}\bigl\{ {Z_{n + r}  = j\bigm| {Z_r  = i} }\bigr\},
    \quad \parbox{2.8cm}{ \textit{for any} {} $r \in {\mathbb{N}}_0$}
$$
    satisfy to the Kolmogorov-Chapman equation
$$
    P_{ij} (n + 1) = \sum\limits_{k \in {\cal S}} {P_{ik} (n)P_{kj} },
    \quad \parbox{2.2cm}{\textit{for} {} $i,j \in {\cal S}$.}
$$
    Hence
\begin{equation}
    \mathbb{E}_i s^{Z_n } : = \sum\limits_{j \in {\cal S}} {P_{ij} (n)s^j } = \bigl[ {F_n (s)} \bigr]^i ,
\end{equation}
    where GF $F_n (s) = \mathbb{E}_1 s^{Z_n }$ is $n$-fold functional iteration
    of $F(s)$; see {\cite[pp.16--17]{Harris66}}.

    Throughout this paper we write $\mathbb{E}$ and $\mathbb{P}$ instead
    of $\mathbb{E}_1 $  and $\mathbb{P}_1 $ respectively.

    It follows from (1.2) that $\mathbb{E}Z_n  = A^n $. The GWP is classified as sub-critical,
    critical and  supercritical, if $A < 1$, $A = 1$ and $A > 1$, accordingly.

    The event $\left\{ {Z_n  = 0} \right\}$ is a simple absorbing state for any GWP.
    The limit $q = \lim _{n \to \infty } P_{10} (n)$ denotes the process starting
    from one individual eventually will be lost and called the extinction probability of GWP.
    It is the least non-negative root of $F(q) = q \le 1$ and that $q = 1$ if the process
    is non-supercritical. Moreover the convergence $  \mathop {\lim }\nolimits_{n \to \infty } F_n (s) = q$
    holds uniformly for $0 \le s \le r < 1$. An assertion describing decrease speed of the function
    $R_n (s): = q - F_n (s)$, due to its importance, is called the Basic Lemma (in fact this name is usually
    used for the critical situation).

    In Section 2 we follow on intentions of papers {\cite{Imomov12}} and {\cite{Imomov14a}} and
    prove an assertion about asymptote of the function $R'_n (s)$ as Differential Analogue of
    Basic Lemma. This simple assertion (and its corollaries, Theorem 1 and 2) will
    lays on the basis of our reasoning in Section 3.

    We start the Section 3 with recalling the Lemma 3 proved in {\cite[p.15]{ANey}}.
    Until the Theorem 6 we study ergodic property of transition functions $\left\{ {P_{ij} (n)} \right\}$,
    having carried out the comparative analysis of known results. We discuss a role
    of $\mu _j = \lim _{n \to \infty } {{P_{1j} (n)} \mathord{\left/  {\vphantom {{P_{1j} (n)}
    {P_{11} (n)}}} \right. \kern-\nulldelimiterspace} {P_{11} (n)}}$ qua the invariant measures and
    seek an analytical form of GF ${\cal M}(s) = \sum\nolimits_{j \in {\cal S}} {\mu _j s^j }$ and
    also we discuss ${\cal R}$-classification of GWP. Further consider the variable ${\cal H}$ denoting
    an extinction time of GWP, that is ${\cal H} = \min \left\{ {n:Z_n  = 0} \right\}$.
    An asymptote of $\mathbb{P} \left\{{{\cal H}= n} \right\}$ has been studied in {\cite{KNS}} and {\cite{Slack}}.
    The event $\left\{ {n < {\cal H} < \infty } \right\}$ represents a condition of $\left\{ Z_n  \ne 0\right\}$ at
    the  moment $n$ and $\left\{ Z_{n + k}  = 0\right\}$ for some $k \in {\mathbb{N}}$. By the extinction theorem
    $\mathbb{P}_i \left\{ {{\cal H} < \infty }\right\} = q^i $. Therefore in non-supercritical case
    $\mathbb{P}_i\left\{ {n < {\cal H} < \infty } \right\}
    \equiv \mathbb{P}_i \left\{{{\cal H} > n} \right\} \to 0$. Hence, $Z_n  \to 0$ with probability one,
    so in these cases the process will eventually die out. We also consider a conditional distribution
$$
    \mathbb{P}_i^{{\cal H}(n)}\{ * \}: = \mathbb{P}_i
    \bigl\{{ * \bigm| {n < {\cal H} < \infty }}\bigr\}.
$$
    in the section. The classical limit theorems state that if $q > 0$ then under certain moment assumptions the
    limit $\widetilde P_{ij} (n): = \mathbb{P}_i^{{\cal H}(n)} \bigl\{ {Z_n  = j} \bigr\}$ exists
    always; see {\cite[p.16]{ANey}}. In particular, Seneta {\cite{Seneta69}} has proved that if $A \ne 1$ then the set
    $\left\{ {\nu _j : = \lim _{n \to \infty } \widetilde P_{1j} (n)} \right\}$
    represents a probability distribution and, limiting GF
    ${\cal V}(s) = \sum\nolimits_{j \in {\cal S}} {\nu _j s^j }$ satisfies to Schroeder equation
\begin{equation}
    1 - {\cal V}\left( {{F(qs)} \over q } \right) = \beta  \cdot \bigl[ {1 - {\cal V}(s)} \bigr],
\end{equation}
    where $\beta  = F'(q)$. The equation (1.3) determines an invariant property of numbers
    $\left\{{\nu _j }\right\}$ with respect to the transition functions
    $\left\{ {\widetilde P_{1j} (n)} \right\}$ and, the set $\left\{ {\nu _j } \right\}$ is
    called ${\cal R}$-invariant measure with parameter ${\cal R} = \beta ^{ - 1}$; see {\cite{Pakes99}}.
    In the critical case we know the Yaglom theorem about a convergence of conditional distribution
    of ${{2Z_n }\mathord{\left/ {\vphantom {{2Z_n } {F''(1)n}}} \right. \kern-\nulldelimiterspace} {F''(1)n}}$ given
    that $\left\{ {{\cal H} > n} \right\}$ to the standard exponential law.
    In the end of the Section we investigate an ergodic property of probabilities $\widetilde P_{ij} (n)$ and
    we refine above mentioned result of Seneta, having explicit form of ${\cal V}(s)$.

    More interesting phenomenon arises if we observe the limit of   $\mathbb{P}_i^{{\cal H}(n + k)} \{  * \} $
    letting $k \to \infty $ and fixed $n \in {\mathbb{N}}$.
    In Section 4 we observe the conditioned limit $\lim _{k \to \infty }
    \mathbb{P}_i^{{\cal H}(n + k)} \bigl\{ {Z_n  = j}\bigr\}$ which represents an honest probability
    measures $\textbf{\textsf{Q}} = \bigl\{ {{\cal Q}_{ij} (n)} \bigr\}$ and defines homogeneous Markov chain called
    the Q-process. Let $W_n $ be the state at the moment $n \in {\mathbb{N}}$ in Q-Process.
    Then $W_0 \mathop  = \limits^d Z_0 $ and $\mathbb{P}_i \bigl\{ {W_n  = j} \bigr\} = {\cal Q}_{ij} (n)$.
    The Q-process was considered first by Lamperti and Ney {\cite{LNey68}}; see, also {\cite[pp.56--60]{ANey}}.
    Some properties of it were discussed by Pakes {\cite{Pakes99}}, {\cite{Pakes71}}, and in
    {\cite{Imomov14b}},  {\cite{Imomov02}}.
    The considerable part of the paper of Klebaner, R\"{o}sler and Sagitov {\cite{KRS07}} is devoted to discussion of
    this process from the viewpoint of branching transformation called the Lamperti-Ney transformation.
    Continuous-time analogue of Q-process was considered by the author {\cite{Imomov12}}.

    Section 5 is devoted to classification properties of Markov chain
    $\bigl\{ {W_n, n \in {\mathbb{N}}}\bigr\}$. Unlike of GWP the Q-process is classified on two
    types depending on value of positive parameter $\beta $. It is positive-recurrent if $\beta <1$
    is transient if $\beta  = 1$. The set $\bigl\{ {\upsilon _j : = \lim _{n \to \infty }
    {{{\cal Q}_{ij} (n)} \mathord{\left/ {\vphantom {{{\cal Q}_{ij} (n)} {{\cal Q}_{i1} (n)}}} \right.
    \kern-\nulldelimiterspace} {{\cal Q}_{i1} (n)}}} \bigr\}$ is an invariant measure for Q-process.
    The section studies properties of the invariant measure.

    Sections 6 and 7 are devoted to examine of structure and long-time behaviors of the total
    state $ S_n  = \sum\nolimits_{k = 0}^{n -1} {W_k }$ in Q-process until time $n$. First we
    consider the joint distribution of the cumulative process $\bigl\{ {W_n ,S_n } \bigr\}$. As a result
    of calculation we will know that in case of $\beta  < 1$ the variables $W_n $ and $S_n $ appear
    asymptotically not dependent. But in the case $\beta  = 1$ we state that under certain conditions the
    normalized cumulative process $\bigl( {{{W_n } \mathord{\left/ {\vphantom {{W_n } {\mathbb{E}W_n }}} \right.
    \kern-\nulldelimiterspace} {\mathbb{E}W_n }};\;{{S_n } \mathord{\left/ {\vphantom {{S_n } {\mathbb{E}S_n }}} \right.
    \kern-\nulldelimiterspace} {\mathbb{E}S_n }}} \bigr)$ weakly converges to the two-dimensional random vector
    having a finite distribution. Comparing results of old researches we note that in case of $\beta  = 1$
    the properties of $S_n $ essentially differ from properties of the total progeny of simple GWP.
    In this connection we refer the reader to {\cite{Dwass69}}, {\cite{KNagaev94}} and {\cite{Kennedy75}}
    in which an interpretation and properties of total progeny of GWP in various contexts was investigated.
    In case of $\beta  < 1$, in accordance with the asymptotic independence property of $W_n$ and $S_n $
    we seek a limiting law of $S_n $ separately. So in Section 7 we state and prove an analogue of Law
    of Large Numbers and the Central Limit Theorem for $S_n $.

\medskip

\section{Basic Lemma and its Differential analogue}

    In this section we observe an asymptotic property of the function $R_n (s): = q - F_n (s)$ and
    its derivative. In the critical situation an asymptotic explicit expansion of this function is
    known from the classical literature which is given in the formula (2.10) below.

    Let $A \ne 1$. First we consider $s \in [0;\,q)$. The mean value theorem gives
\begin{equation}
    R_{n + 1} (s) = F'\bigl( {\xi _n (s)} \bigr)R_n (s),
\end{equation}
    where $\xi _n (s) = q - \theta R_n (s)$, $0 < \theta  < 1$.
    We see that $\xi _n (s) < q$.  Since the GF and its derivatives are monotonically
    non-decreasing then consecutive application of (2.1) leads $R_n (s) < q\beta ^n $.
    Collecting last finding and seeing that $\beta  < 1$ we
    write following inequalities:
\begin{equation}
    F^{(k)} \bigl(q(1 - \beta ^n )\bigr) < F^{(k)} \bigl(\xi _n (s)\bigr) < F^{(k)}(q),
    \quad \parbox{2.3cm}{\textit{for} {} $k = 1,\,2$.}
\end{equation}
    In (2.2) the top index means derivative of a corresponding order. Considering together
    representation (2.1) and inequalities (2.2) we take relations
\begin{equation}
    {{R_{n + 1} (s)} \over \beta } < R_n (s) < {{R_{n + 1} (s)}
    \over {F'\bigl(q(1 - \beta ^n )\bigr)}}  \raise 1.5pt\hbox{.}
\end{equation}

    In turn, by Taylor formula and the iteration for $F(s)$ we have expansion
\begin{equation}
    R_{n + 1} (s) = \beta R_n (s) - {{F''\bigl(\xi _n (s)\bigr)} \over 2}R_n^2(s),
    \quad \parbox{2cm}{\textit{as} {} $n \to \infty$,}
\end{equation}
    where and throughout this section $\xi _n (s)$ is such for which are satisfied relations (2.2).
    Assertions (2.2)--(2.4) yield:
\begin{equation}
    {{F''\bigl(q(1 - \beta ^n )\bigr)} \over {2\beta }} < {\beta
    \over {R_{n +1} (s)}} - {1 \over {R_n (s)}} < {{F''(q)}
    \over {2F'\bigl(q(1 - \beta ^n )\bigr)}}  \raise 1.5pt\hbox{.}
\end{equation}
    Repeated application of (2.5) leads us to the following:
$$
    {1 \over {2\beta }}\sum\limits_{k = 0}^{n- 1} {F''\bigl(q(1 - \beta ^k)\bigr)\beta ^k }
    < {{\beta ^n } \over {R_n (s)}} - {1 \over {q - s}}
    < {{F''(q)} \over 2}\sum\limits_{k = 0}^{n - 1} {{{\beta ^k }
    \over {F'\bigl(q(1 - \beta ^k )\bigr)}}}  \raise 1.5pt\hbox{.}
$$
    Taking limit as $n \to \infty $ from here we have estimation
\begin{equation}
    {{\Delta _1 } \over 2} \le \mathop {\lim }\limits_{n \to \infty }
    \left[ {{{\beta ^n } \over {R_n (s)}} - {1 \over {q - s}}}
    \right]\le {{\Delta _2 } \over 2}  \raise 1.5pt\hbox{,}
\end{equation}
    where
$$
    \Delta _1 : = \sum\limits_{k \in {\mathbb{N}}_0 }
    {{{F''\bigl(q(1 - \beta ^k )\bigr)} \over \beta }\beta ^k }
         \qquad \mbox{\textit{and}} \qquad
    \Delta _2 : = \sum\limits_{k \in {\mathbb{N}}_0 }
    {{{F''(q)} \over {F'\bigl(q(1 - \beta ^k )\bigr)}}\beta ^k }.
$$
    We see that last two series converge. Designating
$$
    {1 \over {A_1 (s)}}: = {1 \over {q - s}} + {{\Delta _1 } \over 2}
        \qquad \mbox{\textit{and}} \qquad
    {1 \over {A_2 (s)}}: = {1 \over {q - s}} + {{\Delta _2} \over 2} \raise 1.5pt\hbox{,}
$$
    we rewrite the relation (2.6) as following:
\begin{equation}
    {1 \over {A_1 (s)}} \le \mathop {\lim }\limits_{n \to \infty }{{\beta ^n }
    \over {R_n (s)}} \le {1 \over {A_2 (s)}}  \raise 1.5pt\hbox{.}
\end{equation}
    Clearly that
$$
    {1 \over {A_2 (s)}} - {1 \over {A_1 (s)}} = {{\Delta _2  - \Delta_1 } \over 2} < \infty.
$$
    So there is a positive $\delta =\delta (s)$ such that $\Delta _1  \le \delta  \le \Delta _2 $ and
    the limit in (2.7) is equal to
\begin{equation}
    {1 \over {{\cal A}(s)}} = {1 \over {q - s}} + {\delta  \over 2}  \raise 1.5pt\hbox{.}
\end{equation}

    Having spent similar reasoning for $s \in [q;\,1)$ as before, we will be convinced that the
    limit $\lim _{n \to \infty } {{\beta ^n } \mathord{\left/ {\vphantom {{\beta ^n } {R_n (s)}}} \right.
    \kern-\nulldelimiterspace} {R_n (s)}} = {\cal A}(s)$ holds for all $s \in [0;1)$.

    So we can formulate the following Basic Lemma.

\begin{lemma}
    The following assertions are true for all $s \in [0;1)$:
\begin{enumerate}
    \item [\textbf{\textsc{(i)}}]  if $A \ne 1$ and $F''(q) < \infty $,  then
\begin{equation}
    R_n (s) = {\cal A}(s) \cdot \beta ^n \left( {1 + o(1)}\right)
    \quad \parbox{2cm}{\textit{as} {} $n \to \infty$,}
\end{equation}
    where the function ${\cal A}(s)$ is defined in (2.8);
    \item [\textbf{\textsc{(ii)}}]  {(\textbf{see {\cite[p.19]{ANey}}})}  if $A = 1$ and $2B: = F''(1) < \infty $, then
\begin{equation}
    R_n (s) = \,{{1 - s} \over {\,(1 - s)Bn + 1}}\left( {1 + o(1)}\right),
    \quad \parbox{2cm}{\textit{as} {} $n \to \infty$,}
\end{equation}
\end{enumerate}
\end{lemma}

    The following lemma is discrete-time analogue of Lemma 2 from {\cite{Imomov14a}}.

\begin{lemma}
    The following assertions hold for all $s \in [0;1)$:
\begin{enumerate}
\item [\textbf{\textsc{(i)}}] if $A\ne 1$ and $F''(q) < \infty $, then
\begin{equation}
    R'_n (s) =  - {\cal K}(s) \cdot \beta ^n \left( {1 + o(1)} \right),
    \quad \parbox{2cm}{\textit{as} {} $n \to \infty$,}
\end{equation}
    where $ {\cal K}(s) = \exp \left\{ { - \delta  \cdot {\cal A}(s)}\right\}$ and
    $\delta  = \delta (s) \in [\Delta _1 ;\,\Delta _2 ]$;

\item [\textbf{\textsc{(ii)}}] if $A = 1$ and $2B: = F''(1) < \infty $, then
\begin{equation}
    R'_n (s) = \,{{\hbar (s)B} \over {\,s - F(s)}}\,R_n^2 (s)\,\left({1 + o(1)} \right),
    \quad \parbox{2cm}{\textit{as} {} $n \to \infty$,}
\end{equation}
    where $ F'(s) \le \hbar (s) \le 1$ and $R_n (s)$ has the expression (2.10).
\end{enumerate}
\end{lemma}

\begin{proof}
    Concerning the first part of the lemma we have equality
\begin{equation}
    {{R'_{n + 1} (s)} \over {R'_n (s)}} = \beta  - F''\bigl(\xi _n (s)\bigr)R_n(s),
\end{equation}
    Let at first $s \in [0;\,q)$. As the function $R_n (s)$ monotonously decreases by $s$,
    then its derivative $R'_n (s) < 0$ and,  hence ${{R'_{n + 1} (s)} \mathord{\left/
    {\vphantom {{R'_{n + 1} (s)} {R'_n (s)}}} \right. \kern-\nulldelimiterspace} {R'_n (s)}} > 0$.
    Therefore, taking the logarithm and after, summarizing along $n$,
    we transform the equality (2.13) to the form of
\begin{equation}
    \ln \left[ { -{{R'_n (s)} \over {\beta ^n }}} \right]
    = \sum\limits_{k = 0}^{n - 1} {\ln \left[ {1 - {{F''\bigl(\xi _k (s)\bigr)}
    \over \beta }R_k (s)} \right]}
    = :\sum\limits_{k = 0}^{n - 1}{\ln L_k (s)},
\end{equation}
    where
$$
    L_n (s) = 1 - {{F''\bigl(\xi _n(s)\bigr)} \over \beta }R_n (s).
$$

    Using elementary inequalities
$$
    {{b - a} \over b} < \ln {b \over a} < {{b - a} \over a}  \raise 1.5pt\hbox{,}
    \quad \parbox{3cm}{\textit{where} {} $0 < b < a$,}
$$
    for $L_k (s)$ (a relevance of the use is easily be checked), we write
\begin{equation}
    {{L_k (s) - 1} \over {L_k (s)}} < \ln L_k (s) < L_k (s) - 1.
\end{equation}
    In accordance with (2.2)
\begin{equation}
     - {{F''(q)} \over \beta }R_k (s) < L_k (s) - 1
     <  - {{F''\bigl(q(1 -\beta ^k )\bigr)} \over \beta }R_k (s) < 0.
\end{equation}
    On the other hand as $R_n (s) < q\cdot \beta ^n $,
    then $F_n (s)> q\cdot \bigl(1 - \beta ^n \bigr)$ and hence
\begin{equation}
    \beta L_k (s) = F'(F_k (s)) > F'\bigl(q(1 - \beta ^k )\bigr).
\end{equation}
    Combining of relations (2.15)--(2.17) yields
$$
     - {{F''(q)} \over {F'\bigl(q(1 - \beta ^k )\bigr)}}R_k (s) < \ln L_k (s)
     < - {{F''\bigl(q(1 - \beta ^k )\bigr)} \over \beta }R_k (s).
$$
    Using this relation in (2.14) we obtain
$$
    \sum\limits_{k = 0}^{n - 1} {{{F''\bigl(q(1 - \beta ^k )\bigr)}
    \over \beta}R_k (s)} < \ln \left[ { - {{\beta ^n }
    \over {R'_n (s)}}}\right] < \sum\limits_{k = 0}^{n - 1} {{{F''(q)}
    \over {F'\bigl(q(1 -\beta ^k )\bigr)}}} R_k (s).
$$
    Hence in our designations
\begin{equation}
    A_2 (s) \cdot \Delta _1  \le \mathop {\lim }\limits_{n \to \infty}
    \ln \left[ { - {{\beta ^n } \over {R'_n (s)}}} \right] \le A_1(s) \cdot \Delta _2,
\end{equation}
    Since $\Delta _1  \le \delta  \le \Delta _2 $, owing to (2.7)--(2.9)
\begin{equation}
    A_2 (s) \le \mathop {\lim }\limits_{n \to \infty } {{R_n (s)}\over {\beta ^n }} = {\cal A}(s) \le A_1 (s).
\end{equation}
    Considering together the estimations (2.18) and (2.19) we conclude
\begin{equation}
    \Delta _1  \le \mathop {\lim }\limits_{n \to \infty } {{\ln \left[{ - {\displaystyle{\beta ^n }
    \over \displaystyle{R'_n (s)}}} \right]} \over {{\cal A}(s)}}\le \Delta _2.
\end{equation}

    The function ${{\beta ^n } \mathord{\left/ {\vphantom {{\beta ^n } {R'_n (s)}}} \right.
    \kern-\nulldelimiterspace} {R'_n (s)}}$ is continuous and monotone by $s$ for each $n \in {\mathbb{N}}_0 $.
    Inequalities (2.20) entail that the functions $\ln \bigl[ { -{{\beta ^n } \mathord{\left/
    {\vphantom {{\beta ^n } {R'_n (s)}}} \right. \kern-\nulldelimiterspace} {R'_n (s)}}} \bigr]$
    converge uniformly for $0 \le s \le z < q$ as $n \to \infty $. From here we get (2.11)
    for $0 \le s < q$. By similar reasoning we will be convinced that convergence (2.11) is fair
    for $s \in [q;\,1)$ and ergo for all values of $s$, such that $0 \le s < 1$.

    Let's prove now the formula (2.12). The Taylor expansion and iteration of $F(s)$ produce
\begin{equation}
    F_n (F(s)) - F_n (s) = \,BR_n^2 (s)\,\left( {1 + o(1)}\right),
    \quad \parbox{2cm}{\textit{as} {} $n \to \infty$.}
\end{equation}
    In the left-side part of (2.21) we apply the mean value Theorem and have
\begin{equation}
    F'_n \left( {c(s)} \right) = \,{B \over {F(s) - s}}\,R_n^2(s)\,\left( {1 + o(1)} \right),
    \quad \parbox{2cm}{\textit{as} {} $n \to \infty$,}
\end{equation}
    where $ s < c(s) < F(s)$. If we use a derivative's monotonicity property of any GF,
    a functional iteration of $F(s)$ entails
$$
    F'_n (s) < F'_n (c(s)) < {{F'_{n + 1} (s)} \over {F'(s)}} \raise 1.5pt\hbox{.}
$$
    From here, using iteration again we have
\begin{equation}
    {{F'(s)} \over {F'\bigl(F_n (s)\bigr)}}F'_n \bigl(c(s)\bigr)
    < F'_n (s) < F'_n \bigl(c(s)\bigr).
\end{equation}
    It follows from relations (2.22), (2.23) and the fact $ F_n (s)\uparrow 1$, that
$$
    F'(s) \le \mathop {\lim }\limits_{n \to \infty }
    {{\bigl( {F(s) - s} \bigr)F'_n (s)} \over {BR_n^2 (s)}} \le 1.
$$
    Designating $\hbar (s)$ the mid-part of last inequalities leads us to the representation (2.12).

    Lemma 2 is proved.
\end{proof}

\begin{remark}
    The function ${\cal A}(s)$ plays the same role, as the akin function in the Basic Lemma for the
    continuous-time Markov branching process established in {\cite{Imomov14a}}; see also {\cite{Imomov12}}.
    Really, it can check up that in the conditions of the Lemma 1,
    $0 < {\cal A}(0) < \infty $, ${\cal A}(q) = 0$, ${\cal A}'(q) =  - 1$, and also it
    is asymptotically satisfied to the Schroeder equation:
$$
    {\cal A}\bigl( {F_n (qs)} \bigr) = \beta ^n \cdot {\cal A}(qs)\bigl( {1 + o(1)} \bigr),
    \quad \parbox{2cm}{\textit{as} {} $n \to \infty$,}
$$
    for all $0 \le s < 1$.
\end{remark}

    Now due to the Lemma 2 we can calculate the probability of return to an initial state $Z_0  = 1$ in time $n$.
    So since $ F'_n (0) = P_{11} (n)$, putting $s = 0$ in (2.11) and (2.12)
    we directly obtain the following two local limit theorems.

\begin{theorem}
    Let $A \ne 1$ and $F''(q) < \infty $. Then
\begin{equation}
    \beta ^{ - n} P_{11} (n) = {\cal K}(0)\left( {1 + o(1)}\right),
    \quad \parbox{2cm}{\textit{as} {} $n \to \infty$,}
\end{equation}
    where the function ${\cal K}(s)$ is defined in (2.11).
\end{theorem}

\begin{theorem}
    If $A = 1$ and the second moment $F''(1) = :2B$ is finite, then
\begin{equation}
    n^2 P_{11} (n) = {{\widehat p_1 } \over {p_0 B}}\left({1 + o(1)} \right),
    \quad \parbox{2cm}{\textit{as} {} $n \to \infty$,}
\end{equation}
    whenever $p_1 \le\widehat p_1  \le 1$.
\end{theorem}

\medskip

\section{An Ergodic behavior of Transition Functions $\left\{ {P_{ij} (n)}\right\}$ and Invariant Measures}

    We devote this section to ergodicity property of transition functions $\left\{ {P_{ij} (n)} \right\}$.
    Herewith we will essentially use the Lemma 2 with combining the following ratio limit
    property (RLP) {\cite{ANey}}.
\begin{lemma}[\textbf{see {\cite[p.15]{ANey}}}]
    If $p_1  \ne 0$, then for all $i,j \in {\cal S}$ the RLP holds:
\begin{equation}
    {{P_{ij} (n)} \over {P_{11} (n)}} \longrightarrow iq^{i - 1} \mu _j,
    \quad \parbox{2cm}{\textit{as} {} $n \to \infty$,}
\end{equation}
    where $\mu _j  = \lim _{n \to \infty } {{P_{1j} (n)}\mathord{\left/ {\vphantom
    {{P_{1j} (n)} {P_{11} (n)}}} \right. \kern-\nulldelimiterspace} {P_{11} (n)}} < \infty$.
\end{lemma}

    Denoting
$$
    {\cal M}_n^{(i)} (s) = \sum\limits_{j \in {\cal S}} {{{P_{ij} (n)}\over {P_{11} (n)}}s^j },
$$
    we see that a GF analogue of assertion (3.1) is
\begin{equation}
    {\cal M}_n^{(i)}(s) \sim iq^{i - 1} {\cal M}_n (s) \longrightarrow iq^{i - 1}{\cal M}(s),
    \quad \parbox{2cm}{\textit{as} {} $n \to \infty$,}
\end{equation}
    here ${\cal M}_n (s) = {\cal M}_n^{(1)} (s)$ and
    ${\cal M}(s) = \sum\nolimits_{j \in {\cal S}} {\mu _j s^j } $.
    The properties of numbers $\left\{ {\mu _j } \right\}$ are of some interest within our purpose.
    In view of their non-negativity the limiting GF ${\cal M}(s)$
    is monotonously not decreasing by $s$. And according to the assertion (3.2) in studying of behavior
    of ${{P_{ij} (n)} \mathord{\left/ {\vphantom {{P_{ij} (n)} {P_{11} (n)}}} \right.
    \kern-\nulldelimiterspace} {P_{11} (n)}}$ is enough to consider function ${\cal M}_n (s)$.

    It has been proved in {\cite[pp.12--14]{ANey}}  the sequence $\left\{ {\mu _j } \right\}$
    satisfies to equation
\begin{equation}
    \beta \mu _j  = \sum\limits_{k \in {\cal S}} {\mu _k P_{kj}},
    \quad \parbox{2.4cm}{ \textit{for all} {} $j \in {\cal S}$,}
\end{equation}
    where $P_{ij}  = \mathbb{P}_i\left\{ {Z_1  = j} \right\}$.
    Therewith the GF ${\cal M}(s)$ satisfies to the functional equation
\begin{equation}
    {\cal M}\bigl( {F(s)} \bigr) = \beta {\cal M}(s) + {\cal M}(p_0),
\end{equation}
    whenever $s$ and $p_0 $ are in the region of convergence of ${\cal M}(s)$.

    The following theorem describes main properties of this function.
\begin{theorem}
    Let $p_1  \ne 0$. Then ${\cal M}(s)$ converges for $0 \le s < 1$. Furthermore
\begin{enumerate}
    \item [\textbf{\textsc{(i)}}]  if $A \ne 1$ and $F''(q) < \infty $, then
\begin{equation}
    {\cal M}(s) = {{{\cal A}(0) - {\cal A}(s)} \over {{\cal K}(0)}} \raise 1.5pt\hbox{,}
\end{equation}
    whenever ${\cal A}(s)$ and ${\cal K}(s)$ are functions in (2.9) and (2.11) respectively;
    \item [\textbf{\textsc{(ii)}}]   if $A = 1$ and $2B: = F''(1) < \infty $,
    then ${\cal M}_n (s) = {\cal M}(s) + r_n (s)$, where
\begin{equation}
    {\cal M}(s) = {{p_0 } \over {\widehat p_1 B}} \cdot {s \over {1 -s}} \raise 1.5pt\hbox{,}
\end{equation}
    and $p_1  \le \widehat p_1  \le 1$, ${r_n (s)} = \mathcal{O}\left( {{1 \mathord{\left/
    {\vphantom {1 n}} \right. \kern-\nulldelimiterspace} n}} \right)$ as $n \to \infty $.
\end{enumerate}
\end{theorem}

\begin{proof}
    The convergence property of GF ${\cal M}(s)$ was proved in {\cite[p.13]{ANey}}.

    In our designations we write
\begin{equation}
    {\cal M}_n (s) = {{F_n (s) - F_n (0)} \over {F'_n (0)}}
    = \left({1 - {{R_n (s)} \over {R_n (0)}}} \right) \cdot {{R_n (0)} \over{P_{11} (n)}} \raise 1.5pt\hbox{.}
\end{equation}
    In case $A \ne 1$ it follows from (2.9) that
$$
    {{R_n (s)} \over {R_n (0)}} \longrightarrow {{{\cal A}(s)} \over {{\cal A}(0)}} \raise 1.5pt\hbox{,}
    \quad \parbox{2cm}{\textit{as} {} $n \to \infty$,}
$$
    and, considering (2.24) implies
\begin{equation}
    {{R_n (0)} \over {P_{11} (n)}} \longrightarrow {{{\cal A}(0)} \over {{\cal K}(0)}} \raise 1.5pt\hbox{.}
\end{equation}
    Combining (3.7) and (3.8) we obtain ${\cal M}(s)$ in form of (3.5).

    Let's pass to the case $A = 1$. Due to statement of (2.10) appears
\begin{equation}
    1 - {{R_n (s)} \over {R_n (0)}} \sim \,{s \over {\,(1 - s)Bn +1}} \raise 1.5pt\hbox{,}
    \quad \parbox{2cm}{\textit{as} {} $n \to \infty$.}
\end{equation}
    In turn according to (2.25)
\begin{equation}
    {{R_n (0)} \over {P_{11} (n)}} \sim {{\,p_0 } \over {\widehat p_1}}\,n,
    \quad \parbox{2cm}{\textit{as} {} $n \to \infty$.}
\end{equation}
    Considering together relations (3.7), (3.9) and (3.10) we obtain
$$
    {\cal M}_n (s) \sim {{\,p_0 } \over {\widehat p_1 }}{{sn} \over{\,(1 - s)Bn + 1}}\raise 1.5pt\hbox{,}
    \quad \parbox{2cm}{\textit{as} {} $n \to \infty$.}
$$
    Taking limit from here we find the limiting GF in the form of (3.6).

    The proof is completed.
\end{proof}

\begin{remark}
    The theorem above is an enhanced form of Theorem 2 from {\cite[p.13]{ANey}} in
    sense that in our case we get the information on analytical form of limiting GF ${\cal M}(s)$.
\end{remark}

    The following assertions follow from the theorem proved above.

\begin{corollary}
    Let $p_1  \ne 0$. Then
\begin{enumerate}
    \item [\textbf{\textsc{(i)}}]  if $A \ne 1$ and $F''(q) < \infty $, then
\begin{equation}
    {\cal M}(q) = \sum\limits_{j \in {\cal S}} {\mu _j q^j }
    ={{{\cal A}(0)} \over {{\cal K}(0)}} < \infty ;
\end{equation}
    \item [\textbf{\textsc{(ii)}}]   if $A = 1$ and $2B:= F''(1) < \infty $, then
\begin{equation}
    \sum\limits_{j = 1}^n {\mu _j } \sim {{p_0 } \over {\widehat p_1 B}}\,n,
    \quad \parbox{2cm}{\textit{as} {} $n \to \infty$.}
\end{equation}
\end{enumerate}
\end{corollary}

\begin{proof}
    The relation (3.11) follows from (3.5).
    In case $A = 1$ as shown in (3.6)
$$
    {\cal M}(s) \sim {{p_0 } \over {\widehat p_1 B}} \cdot {1 \over {1- s}}  \raise 1.5pt\hbox{,}
    \quad \parbox{1.6cm}{\textit{as} {} $s \uparrow 1$.}
$$
    According to the Hardy-Littlewood Tauberian theorem the last relation entails (3.12).
\end{proof}

    Now from the Lemma 3 and Theorems 1 and 2 we get complete account about asymptotic
    behaviors of transition functions $P_{ij} (n)$. Following theorems are fair.

\begin{theorem}
    Let $p_1  \ne 0$. If $A \ne 1$ and $F''(q) < \infty $, then
$$
    \beta ^{ - n} P_{ij} (n) = {{{\cal A}(0)} \over {{\cal M}(q)}}iq^{i - 1}
    \mu _j \left( {1 + o(1)} \right), \quad \parbox{2cm}{\textit{as} {} $n \to \infty$.}
$$
\end{theorem}

\begin{theorem}
    Let $p_1  \ne 0$. If in critical GWP the second moment $F''(1) = :2B$ is finite then for transition
    functions the following asymptotic representation holds:
$$n
    ^2 P_{ij} (n) = {{\widehat p_1 } \over {p_0 B}}i\mu _j \left( {1+ o(1)} \right),
    \quad \parbox{2cm}{\textit{as} {} $n \to \infty$.}
$$
\end{theorem}

    Further we will discuss the role of the set $\left\{ {\mu _j }\right\}$ as invariant measures concerning
    transition probabilities $\left\{ {P_{ij} (n)} \right\}$. An invariant (or stationary) measure of
    the GWP is a set of nonnegative numbers $\left\{ {\mu _j^ *} \right\}$ satisfying to equation
\begin{equation}
    \mu _j^ *  = \sum\limits_{k \in {\cal S}} {\mu _k^ *  P_{kj} }.
\end{equation}
    If $\sum\nolimits_{j \in {\cal S}} {\mu _j^ * }  < \infty $ (or without loss of generality
    $\sum\nolimits_{j \in {\cal S}} {\mu _j^ *  }  = 1$) then it is called as invariant distribution.
    As $P_{00} (n) = 1$ then according to (3.13) $\mu _0^ *   = 0$ for any invariant measure
    $\left\{ {\mu _j^ *  } \right\}$. If $P_{10} (n) = 0$ then condition (3.13) becomes
    $\mu _j^ * = \sum\nolimits_{k = 1}^j {\mu _k^ *  P_{kj} (n)} $.
    If $P_{10} (n) > 0$ then $P_{i0} (n) > 0$ and  hence $\mu _j^ *   > 0$.

    In virtue of Theorem 4 in non-critical situation the transition functions
    $P_{ij} (n)$ exponentially decrease to zero as $n \to \infty$.
    Following a classification of the continuous-time Markov process
    we characterize this decrease by a "decay parameter"
$$
    {\cal R} =  - \mathop {\lim }\limits_{n \to \infty }{{\ln P_{ii}(n)} \over n}  \raise 1.5pt\hbox{.}
$$
    We classify the non-critical Markov chain $\left\{ {Z_n ,\,\,n \in {\mathbb{N}}_0 } \right\}$
    as ${\cal R}$-transient if
$$
    \sum\limits_{n \in {\mathbb{N}}} {e^{{\cal R}n}P_{ii} (n)}  < \infty
$$
    and ${\cal R}$-recurrent otherwise. This chain is called as ${\cal R}$-positive if
    $\lim _{n \to \infty } e^{{\cal R}n} P_{ii} (n) > 0$, and ${\cal R}$ -null if last limit
    is equal to zero.

    Now assertion(3.11) and Theorem 4 yield the following statement.

\begin{theorem}
    Let $p_1  \ne 0$. If $A \ne 1$ and $F''(q) < \infty $, then ${\cal R} = \left| {\ln \beta } \right|$ and
    the chain $\left\{ {Z_n } \right\}$ is ${\cal R}$-positive. The set of numbers $\left\{ {\mu _j } \right\}$
    determined by GF (3.5) is the unique (up to multiplicative constant) ${\cal R}$-invariant measure for GWP.
\end{theorem}

    In critical situation the set $\left\{ {\mu _j } \right\}$ directly enters to a role of invariant measure
    for the GWP. Indeed, in this case $\beta  = 1$ and according to (3.3) the following invariant equation holds:
$$
    \mu _j  = \sum\limits_{k \in {\cal S}} {\mu _k P_{kj} },
    \quad \parbox{2.4cm}{\textit{for all} {} $j \in {\cal S}$,}
$$
    and owing to (3.12) $\sum\nolimits_{j \in {\cal S}} {\mu _j } = \infty $ .

\begin{remark}
    As shown in Theorems 4 and 5 hit probabilities of GWP to any states through the long interval time
    depend on the initial state. That is ergodic property for
    $\left\{ {Z_n , n \in {\mathbb{N}}_0 } \right\}$ is not carried out.
\end{remark}

    Our further reasoning is connected with earlier introduced variable
$$
    {\cal H}: = \min \bigl\{ {n \in {\mathbb{N}}:\;Z_n  =0} \bigr\},
$$
    which denote the extinction time of GWP. Let  as before
$$
    \mathbb{P}_i^{{\cal H}(n)} \{* \}:
    = \mathbb{P}_i \bigl\{{* \bigm|{n < {\cal H}< \infty }} \bigr\}.
$$
    Put into consideration probabilities $\widetilde P_{ij} (n)=
    \mathbb{P}_i^{{\cal H}(n)} \bigl\{ {Z_n  = j} \bigr\}$ and denote
$$
    {\cal V}_n^{(i)} (s) = \sum\limits_{j \in {\cal S}} {\widetilde P_{ij} (n)s^j }
$$
    to be the appropriate GF. As it has been noticed in the introduction section that if $q > 0$,
    then the limit $\nu _j : =\lim _{n \to \infty } \widetilde P_{1j} (n)$ always exists.
    In case of $A \ne 1$ the set $\left\{ {\nu _j } \right\}$ represents a probability distribution.
    And limiting GF ${\cal V}(s) = \sum\nolimits_{j \in {\cal S}} {\nu _j s^j } $ satisfies to
    Schroeder's equation (1.3) for $0 \le s \le 1$. But if $A = 1$ then $\nu _j  \equiv 0$;
    see {\cite{Seneta69}} and {\cite[p.16]{ANey}}. In forthcoming two theorems we observe
    the limit of $\widetilde P_{ij} (n)$ as $n \to \infty $ for any $i,j \in {\cal S}$.
    Unlike aforementioned results of Seneta we get the explicit expressions for the appropriate GF.

\begin{theorem}
    Let $p_1  \ne 0$. If $A \ne 1$ and $F''(q) < \infty $, then
$$
    \mathop {\lim }\limits_{n \to \infty } \widetilde P_{ij} (n) = \nu_j,
    \quad \parbox{2.4cm}{\textit{for all} {} $j \in {\cal S}$,}
$$
    and suitable GF ${\cal V}(s) =\sum\nolimits_{j \in {\cal S}} {\nu _j s^j } $ has a form of
\begin{equation}
    {\cal V}(s) = 1 - {{{\cal A}(qs)} \over {{\cal A}(0)}}  \raise 1.5pt\hbox{,}
\end{equation}
    where the function ${\cal A}(s)$ is defined in (2.8).
\end{theorem}

\begin{proof}
    We write
\begin{equation}
    \widetilde P_{ij} (n) = {{\mathbb{P}_i \bigl\{ {Z_n  = j,\; n <{\cal H} < \infty } \bigr\}}
    \over {\mathbb{P}_i \bigl\{ {n <{\cal H} < \infty } \bigr\}}}  \raise 1.5pt\hbox{.}
\end{equation}
    In turn
$$
    \mathbb{P}_i \bigl\{ {Z_n  = j,\; n < {\cal H} < \infty } \bigr\}
    = \mathbb{P}\bigl\{ {n < {\cal H} < \infty \bigm| {Z_n  = j} }\bigr\} \cdot P_{ij} (n).
$$
    Since the vanishing probability of $j$ particles is equal
    to $q^j $ then from last form we receive that
\begin{equation}
    \mathbb{P}_i \bigl\{ {Z_n  = j,\; n < {\cal H} < \infty }\bigr\} = q^j \cdot P_{ij} (n)
\end{equation}
    Using relation (3.16) implies
\begin{equation}
    \mathbb{P}_i \bigl\{ {n < {\cal H} < \infty } \bigr\} = \sum\limits_{j\in {\cal S}}
    {\mathbb{P}_i \bigl\{ {Z_n  = j,\; n < {\cal H} < \infty }\bigr\}}
    = \sum\limits_{j \in {\cal S}} {P_{ij} (n)q^j }.
\end{equation}

    Now it follows from (3.15)--(3.17) and Lemma 3 that
$$
    \widetilde P_{ij} (n) = {\displaystyle{{{P_{ij} (n)} \over {P_{11} (n)}} \cdot q^j } \over
    {\displaystyle \sum\nolimits_{k \in {\cal S}} {{{P_{ik} (n)} \over {P_{11} (n)}}q^k } }}\; \buildrel {}
    \over \longrightarrow \;{{\mu _j \cdot q^j} \over {\sum\nolimits_{k \in{\cal S}}{\mu_k q^k }}}
    = {{\mu _j q^j } \over {{\cal M}(q)}} = :\nu _j,
$$
    as $n \to \infty $. It can be verified the limit distribution  $\left\{ {\nu _j } \right\}$ defines the
    GF  ${\cal V}(s) = {{{\cal M}(qs)} \mathord{\left/ {\vphantom {{{\cal M}(qs)} {{\cal M}(q)}}} \right.
    \kern-\nulldelimiterspace} {{\cal M}(q)}}$. Applying here equality (3.5) we get to (3.14).
\end{proof}

\begin{remark}
    The mean of distribution measure $\widetilde P_{ij} (n)$
$$
    \sum\limits_{j \in {\cal S}} {j\widetilde P_{ij} (n)}  \longrightarrow {q\over {{\cal A}(0)}} \raise 1.5pt\hbox{,}
    \quad \parbox{2cm}{\textit{as} {} $n \to \infty$}
$$
    and, the limit distribution $ \left\{ {\nu _j } \right\}$ has the finite
    mean ${\cal V}'(s \uparrow 1) = {q \mathord{\left/ {\vphantom {q {{\cal A}(0)}}}
    \right. \kern-\nulldelimiterspace} {{\cal A}(0)}}$.
\end{remark}

    Further consider the case $A = 1$. In this case
    $\mathbb{P}\left\{ {{\cal H} < \infty } \right\} = 1$, therefore
\begin{eqnarray}
    {\cal V}_n^{(i)} (s) \nonumber
    & = & \sum\limits_{j \in {\cal S}} {\mathbb{P}_i \bigl\{ {Z_n  = j\bigm| {{\cal H} > n}} \bigr\}s^j } \\
    & = & \sum\limits_{j \in {\cal S}} {{{P_{ij} (n)} \over {\mathbb{P}_i \bigl\{ {Z_n  > 0} \bigr\}}}s^j }
    = 1 - {{1 - F_n^i (s)} \over {1 - F_n^i (0)}}  \raise 1.5pt\hbox{.} \nonumber
\end{eqnarray}
    We see that $ 1 - F_n^i (s) \sim iR_n (s)$ as $n \to \infty $. Hence considering (3.7) obtains
\begin{equation}
    {\cal V}_n^{(i)} (s) \sim 1 - {{R_n (s)} \over {R_n (0)}}
    ={{P_{11} (n)} \over {R_n (0)}} \cdot {\cal M}_n(s),
    \quad \parbox{2cm}{\textit{as} {} $n \to \infty$.}
\end{equation}
    Combining expansions (2.10), (2.25), (3.6) and (3.18), we state the following theorem.

\begin{theorem}
    Let $A = 1$. If $2B: = F''(1) < \infty $, then
$$
    n{\cal V}_n^{(i)} (s) = {1 \over B} \cdot {s \over {1 - s}} + \rho_n (s),
$$
    where $ {\rho _n (s)} = \mathcal{O}\left( {{1 \mathord{\left/
    {\vphantom {1 n}} \right. \kern-\nulldelimiterspace} n}} \right)$ as $n \to \infty $.
\end{theorem}
\begin{remark}
    It is a curious fact that in last theorem we managed to be saved of undefined variable
    $\widehat p_1  \in [p_1 ;1]$.
\end{remark}

    Now define the stochastic process $\widetilde Z_n $ with the transition matrix
    $\left\{ {\widetilde P_{ij} (n)} \right\}$. It is easy to be convinced that $\widetilde Z_n $ represents
    a discrete-time Markov chain. According to last theorems the properties of its trajectory lose
    independence on initial state with growth the numbers of generations.

    In non-critical case, according to the Theorem 7, for GWP $\widetilde Z_n $ there
    is (up to multiplicative constant) unique set of nonnegative numbers $\left\{ {\nu _j } \right\}$ which
    are not all zero and $\sum\nolimits_{j \in {\cal S}} {\nu _j }  = 1$.
    Moreover as $ {\cal M}(qs) = {\cal M}(q) \cdot {\cal V}(s)$ then using the
    formula (3.4) we can establish the following invariant equation:
$$
    \beta \cdot {\cal V}(s) = {\cal V}\left( {\widehat F(s)} \right) - {\cal V}\left( {\widehat F(s)} \right),
$$
    where ${\cal V}(s) =\sum\nolimits_{j \in {\cal S}} {\nu _j s^j } $ and
    $\widehat F(s) ={{F(qs)} \mathord{\left/ {\vphantom {{F(qs)} q}} \right. \kern-\nulldelimiterspace} q}$.

    So we have the following

\begin{theorem}
    Let $A \ne 1$ and $F''(q) < \infty $.  Then
$$
    P_{ij} (n) = \widetilde P_{ij} (n) \cdot \sum\limits_{k\in {\cal S}} {P_{ik} (n)q^{k - j} },
$$
    where transition functions $\widetilde P_{ij} (n)$ have an ergodic property and their limits
    $\nu _j  = \lim _{n \to \infty } \widetilde P_{ij} (n)$ present $\left| {\ln \beta } \right|$-invariant
    distribution for the Markov chain $\left\{ {\widetilde Z_n } \right\}$.
\end{theorem}

    In critical situation we have the following assertion which directly implies from
    Theorem 8 and taking into account the continuity theorem for GF.

\begin{theorem}
    If in critical GWP $2B: = F''(1) < \infty $, then
$$
    n\widetilde P_{ij} (n) = {1 \over B} + \mathcal{O}\left( {{\,1 \over \,n}}\right),
    \quad \parbox{2cm}{\textit{as} {} $n \to \infty$.}
$$
\end{theorem}

\medskip

\section{Limiting interpretation of $\mathbb{P}_i^{{\cal H}(n + k)}\{* \}$}

    In this section, excepting cases $p_1 = 0$ and $q = 0$, we observe the
    distribution $\mathbb{P}_i^{{\cal H}(n + k)} \{ Z_n  = j\}$. It has still
    been noticed by  Harris {\cite{Harris51}} that its limit as $k\to \infty $
    always exists for any fixed $n \in {\mathbb{N}}$. By means of
    relations (3.15)--(3.17) it was obtained in {\cite[pp.56--60]{ANey}} that
$$
    \mathop {\lim }\limits_{k \to \infty } \mathbb{P}_i^{{\cal H}(n + k)} \bigl\{Z_n  = j\bigr\}
    = {{jq^{j - i} } \over {i\beta ^n }}P_{ij} (n) = :{\cal Q}_{ij} (n).
$$
    Since $F'_n (q) = \left[ {F'(q)} \right]^n= \beta ^n $, then by (1.2)
$$
    \sum\limits_{j \in {\cal S}} {{{jq^{j - i} } \over {i\beta ^n}}P_{ij} (n)} = {1 \over {iq^{i - 1}
    \beta ^n }}\left[{\sum\nolimits_{j \in {\cal S}} {P_{ij} (n)s^j } } \right]^\prime_{s = q}  = 1.
$$
    So we have an honest probability measure $\textbf{\textsf{Q}}
    =\left\{ {{\cal Q}_{ij} (n)} \right\}$.  The stochastic process
    $\left\{ {W_n , n \in {\mathbb{N}}_0 } \right\}$ defined by this measure is called the Q-process.

    By definition
$$
    \textbf{\textsf{Q}} = \left\{{\lim _{k \to \infty} \mathbb{P}_i
    \bigl\{{* \bigm| {n + k < {\cal H} < \infty }} \bigr\}}\right\}
    = \bigl\{ {\mathbb{P}_i \bigl\{{* \bigm|{{\cal H}
    = \infty}}\bigr\}} \bigr\},
$$
    that the Q-process can be considered as GWP with a non-degenerating trajectory in remote
    future, that is  it conditioned on event $\left\{ {{\cal H} = \infty } \right\}$.
    Harris {\cite{Harris51}} has established that if $A = 1$ and
    $2B: = F''(1) < \infty $ the distribution of ${{Z_n } \mathord{\left/ {\vphantom {{Z_n } {Bn}}} \right.
    \kern-\nulldelimiterspace} {Bn}}$ conditioned on $\left\{ {{\cal H} = \infty } \right\}$ has the
    limiting Erlang's law. Thus the Q-process $\left\{ {W_n , n \in {\mathbb{N}}_0 } \right\}$ represents
    a homogeneous Markov chain with initial state $W_0 \mathop  = \limits^d Z_0 $ and
    general state space which will henceforth denoted as ${\cal E} \subset {\mathbb{N}}$.
    The variable $W_n $ denote the state size of this chain in instant $n$ with the transition matrix
\begin{equation}
    {\cal Q}_{ij} (n) = \mathbb{P}_i \bigl\{ {W_{n + k}  = j} \bigr\}
    = {{jq^{j - i} } \over {i\beta ^n }}P_{ij} (n),
    \quad \parbox{2.7cm}{\textit{for all} {} $i, j \in {\cal E}$,}
\end{equation}
    and for any $n,k \in {\mathbb{N}}$ .

    Put into consideration a GF
$$
    Y_n^{(i)} (s): = \sum\limits_{j \in {\cal E}} {{\cal Q}_{ij}(n)s^j }.
$$
    From (1.2) and (4.1) we have
\begin{eqnarray}
    Y_n^{(i)} (s) \nonumber
    & = & \sum\limits_{j \in {\cal E}} {{{jq^{j - i} } \over {i\beta ^n }}P_{ij} (n)s^j } \\
    & = & {{q^{1 - i} s} \over {i\beta ^n }}\sum\limits_{j \in {\cal E}} {P_{ij} (n)(qs)^{j - 1} }
    = {{qs} \over {i\beta ^n }}{\partial  \over {\partial x}}\left[ {\left( {{{F_n (x)}
    \over q}} \right)^i } \right]_{x = qs}. \nonumber
\end{eqnarray}
    Therefore
\begin{equation}
    Y_n^{(i)} (s) = \left[ {{{F_n (qs)} \over q}} \right]^{i - 1} Y_n(s),
\end{equation}
    where GF $Y_n (s): = Y_n^{(1)} (s) = \mathbb{E}\left[{s^{W_n }
    \left| {W_0  = 1} \right.} \right]$ has the form of
\begin{equation}
    Y_n (s) = s{{F'_n (qs)} \over {\beta ^n }}  \raise 1.5pt\hbox{,}
    \quad \parbox{2.4cm}{\textit{for all} {} $n \in  {\mathbb{N}}$.}
\end{equation}

    As $F_n (s) \to q$ owing to (4.2) and (4.3), ${{{\cal Q}_{ij} (n)} \mathord{\left/
    {\vphantom {{{\cal Q}_{ij}(n)}{{\cal Q}_{1j}(n)}}}\right. \kern-\nulldelimiterspace}
    {{\cal Q}_{1j} (n)}} \to 1$, at infinite growth of the number of generations.
    Using (4.2) and iterating $F(s)$ produce a following functional relation:
\begin{equation}
    Y_{n + 1}^{(i)} (s) = {{Y(s)} \over {\widehat F(s)}}Y_n^{(i)}\left( {\widehat F(s)} \right),
\end{equation}
    where $\widehat F(s) = {{F(qs)}\mathord{\left/{\vphantom {{F(qs)}q}}\right.
    \kern-\nulldelimiterspace} q}$ and $Y(s): = Y_1 (s)$.
    We see that Q-process is completely defined by GF
$$
    Y(s) = s{{F'(qs)} \over \beta }
$$
    and, its evolution is regulated by the positive parameter $\beta $. In fact, if the first
    moment  $\alpha : = Y'(1)$ is finite then differentiating of (4.3) in $s = 1$ gives
$$
    \mathbb{E}_i W_n  = \left( {i - 1} \right)\beta ^n  + \mathbb{E}W_n
$$
    and
\begin{equation}
    \mathbb{E}W_n  =\left\{\begin{array}{l} 1 + \gamma \left( {1 - \beta ^n } \right) \, \hfill,
    \qquad \parbox{2.1cm}{\textit{when} {} $\beta < 1 $,}  \\
    \\
    \left( {\alpha  - 1} \right)n + 1 \hfill, \qquad  \parbox{2.1cm}{\textit{when} {} $\beta = 1 $,}  \\
    \end{array} \right.
\end{equation}
    where $\gamma := {{\left( {\alpha - 1} \right)} \mathord{\left/ {\vphantom {{\left({\alpha - 1}
    \right)} {\left( {1 - \beta } \right)}}} \right. \kern-\nulldelimiterspace}
    {\left( {1 - \beta } \right)}}$ and $\alpha  = 1 + {{\widehat F^{''} (1)} \mathord{\left/
    {\vphantom {{\widehat F^{''}  (1)} \beta }} \right. \kern-\nulldelimiterspace} \beta } > 1$.

\medskip

\section{Classification and ergodic behavior of states of Q-processes}

    The formula (4.5) shows that if $\beta  < 1$, then
$$
    \mathbb{E}_i W_n  \longrightarrow 1 + \gamma , \quad \parbox{2cm}{\textit{as} {} $n \to \infty$}
$$
    and, provided that $\beta  = 1$
$$
    \mathbb{E}_i W_n  \sim \left( {\alpha  - 1} \right)n, \quad \parbox{2cm}{\textit{as} {} $ n \to \infty $.}
$$
    The Q-Process has the following properties:
\begin{enumerate}
    \item [\textbf{\textsc{(i)}}]  if $\beta  < 1$, then it is positive-recurrent;
    \item [\textbf{\textsc{(ii)}}]  if $\beta  = 1$, then it is transient.
\end{enumerate}
    In the transient case $W_n  \to \infty $ with probability $1$; see {\cite[p.59]{ANey}}.

    Let's consider first the positive-recurrent case. In this case according to (2.11), (4.2), (4.3) the
    limit $\pi (s): = \lim _{n\to \infty } Y_n^{(i)} (s)$ exists provided that $\alpha  < \infty $.
    Then owing to (4.4) we make sure that GF $\pi (s) =\sum\nolimits_{j \in {\cal E}} {\pi _j s^j } $
    satisfies to invariant equation $\pi (s){{ \cdot F(qs)} \mathord{\left/ {\vphantom {{ \cdot F(qs)} q}} \right.
    \kern-\nulldelimiterspace} q} = Y(s) \cdot \pi \left({{{F(qs)}\mathord{\left/{\vphantom {{F(qs)} q}}\right.
    \kern-\nulldelimiterspace} q}} \right)$. Applying this equation reduces to
\begin{equation}
    \pi (s) = {{Y_n (s)} \over{\widehat{F_n }(s)}}\pi \left( {\widehat{F_n }(s)} \right),
\end{equation}
    where $ \widehat{F_n }(s) = {{F_n (qs)} \mathord{\left/ {\vphantom {{F_n (qs)} q}} \right.
    \kern-\nulldelimiterspace} q}$.  A transition function analogue of (5.1) is form of
    $\pi _j  = \sum\nolimits_{i \in {\cal E}} {\pi _i {\cal Q}_{ij}(n)}$. Taking limit in (5.1)
    as $n \to \infty $ it follows that $\pi \left( {\widehat{F_n }(s)} \right) \sim \widehat{F_n }(s)$ and
    it in turn entails $\sum\nolimits_{j \in {\cal E}} {\pi _j }  = 1$ since $\widehat{F_n }(s) \to 1$.
    So in this case the set $\left\{ {\pi _j , j \in {\cal E}}\right\}$ represents an invariant distribution.
    Differentiation (5.1) and taking into account (4.5) we  easily compute that
\begin{equation}
    \pi '(1) = \sum\nolimits_{j \in {\cal E}} {j\pi _j }  = 1 + \gamma,
\end{equation}
    where as before $\gamma : = {{\left( {\alpha  - 1} \right)} \mathord{\left/ {\vphantom
    {{\left( {\alpha  - 1} \right)} {\left( {1 - \beta } \right)}}} \right.
    \kern-\nulldelimiterspace} {\left( {1 - \beta } \right)}}$.

    Further we note that owing to (2.11) and (4.2)
$$
    \pi (s) = s\exp \bigl\{ { - \delta (qs) \cdot {\cal A}(qs)}\bigr\},
$$
    where the function ${\cal A}(s)$ looks like (2.8). Since $\pi (1) = 1$ and
     ${\cal A}(qs) = \mathcal{O}\left( {1 - s} \right)$ as $s \uparrow 1$ it is necessary to be
$$
    \delta (qs) = \mathcal{O}\left( {(1 - s)^{ - \sigma } } \right)
$$
    with $\sigma  < 1$. On the other hand for feasibility of equality (5.2) is equivalent to that
$$
    \left. {{{\partial \bigl[ {\delta (qs) \cdot {\cal A}(qs)}\bigr]}
    \over {\partial s}}} \right|_{s \uparrow 1}  =  - \gamma.
$$
    If we remember the form of function ${\cal A}(s)$ the last condition becomes
\begin{equation}
    \mathop {\lim }\limits_{s \uparrow 1} \left\{ {\delta'(qs)\left[ {q (1 - s)
    - {{\delta (qs)} \over 2}q^2 (1 - s)^2 }\right] - q\delta (qs)} \right\} =  - \gamma.
\end{equation}
    For the function $ \delta  = \delta (s)$ all cases are disregarded except for
    the unique case $\sigma  = 0$ for the following simple reason. All functions having a form
    of $(1 - s)^{ - \sigma }$ monotonically increase to infinity as $s \uparrow 1$ when $0 < \sigma  < 1$
    and this fact contradicts the boundedness of function $\delta  = \delta (s)$.
    In the case $ \sigma  < 0$ cannot be occurred (5.3) since the limit in the left-hand part is
    equal to zero while $\gamma  \ne 0$. In unique case $ \sigma  = 0$ the limit
    is constant and in view of (5.3)
$$
    \delta  = {\gamma  \over q}  \raise 1.5pt\hbox{.}
$$

    We proved the following theorem.

\begin{theorem}
    If $\beta  < 1$ and $\alpha : = Y'(1) < \infty $, then for $0 \le s < 1$
\begin{equation}
    \mathop {\lim }\limits_{n \to \infty } Y_n^{(i)} (s) = \pi (s),
\end{equation}
    where $\pi (s)$ is probability GF having a form of
$$
    \pi (s) = s\exp \left\{ { - {{\gamma (1 - s)} \over {1 + {\displaystyle \gamma
    \over \displaystyle 2}(1 - s)}}} \right\}.
$$
\end{theorem}

    The set $ \left\{ {\pi _j , j \in {\cal E}} \right\}$ coefficients in power series expansion
    of $\pi (s) = \sum\nolimits_{j \in {\cal E}} {\pi _j s^j}$ are invariant distribution for the Q-process.

    In transient case the following theorem hold.

\begin{theorem}
    If $\beta  = 1$ and $\alpha : = Y'(1) < \infty $, then for all $0 \le s < 1$
\begin{equation}
    n^2 Y_n^{(i)} (s) = \mu (s)\left( {1 + r_n (s)} \right),
    \quad \parbox{2cm}{\textit{as} {} $ n \to \infty $,}
\end{equation}
    where $ {r_n (s)}  = o(1)$ for $0 \le s < 1$ and the GF
    $\mu (s) = \sum\nolimits_{j \in {\cal E}} {\mu _j s^j }$ has a form of
$$
    \mu (s) = {{2s\hbar (s)}
    \over {(\alpha  - 1)\bigl( {F(s) - s}\bigr)}} \raise 1.5pt\hbox{,}
$$
    with $ Y(s) \le s\hbar (s) \le s$. Nonnegative numbers
    $\left\{ {\mu _j, j \in {\cal E}} \right\}$ satisfy to invariant equation
\begin{equation}
    \mu _j  = \sum\nolimits_{i \in {\cal E}} {\mu _i {\cal Q}_{ij}(n)}.
\end{equation}
    Moreover $\sum\nolimits_{j \in {\cal E}} {\mu_j }  = \infty $.
\end{theorem}

\begin{proof}
    The convergence (5.5) immediately follows as a result of combination of (2.12), (4.2) and (4.3).
    Taking limit in (4.4) reduces to equation $\mu (s)F_n (s) = Y_n (s)\mu \left( {F_n (s)} \right)$ which
    equivalent to (5.6) in the context of transition probabilities. On the other hand it
    follows from (5.5) that $\mu \left( {F_n (s)} \right) \sim n^2 F_n (s)$ as $n \to \infty $.
    Hence $\sum\nolimits_{j \in {\cal E}} {\mu _j }  = \infty $ .
\end{proof}

    As $\lim _{s \downarrow 0} \left[ {{{Y_n^{(i)} (s)}\mathord{\left/ {\vphantom
    {{Y_n^{(i)} (s)} s}} \right. \kern-\nulldelimiterspace} s}} \right] = {\cal Q}_{i1} (n)$,
    the following two theorems imply from (5.4) and (5.5).

\begin{corollary}
    If $\beta  < 1$ and $\alpha : = Y'(1) < \infty $, then
\begin{equation}
    {\cal Q}_{i1} (n) = e^{ - {{2\gamma } \mathord{\left/ {\vphantom {{2\gamma }
    {(2 + \gamma )}}} \right. \kern-\nulldelimiterspace} {(2 + \gamma )}}} \left( {1 + o(1)} \right),
    \quad \parbox{2cm}{\textit{as} {} $ n \to \infty $.}
\end{equation}
\end{corollary}

\begin{corollary}
    If $\beta  = 1$ and $\alpha : = Y'(1) < \infty $, then
\begin{equation}
    n^2 {\cal Q}_{i1} (n) = {{2\widetilde{\cal Q}_1 } \over{(\alpha  - 1)p_0 }}\left( {1 + o(1)} \right),
    \quad \parbox{2cm}{\textit{as} {} $ n \to \infty $,}
\end{equation}
    here ${\cal Q}_{11}(1) \le \widetilde{\cal Q}_1  \le 1$.
\end{corollary}

\begin{theorem}
    Let $\beta  = 1$ and $\alpha : = Y'(1) < \infty $. Then
\begin{equation}
    \mathop {\lim }\limits_{n \to \infty } {1 \over {n^2 }}\left[ {\mu_1  + \mu _2
    +  \cdots  + \mu _n } \right] = {2 \over {\left({\alpha  - 1} \right)^2 }}  \raise 1.5pt\hbox{.}
\end{equation}
\end{theorem}

\begin{proof}
    By Taylor formula $F(s) - s \sim B(1 - s)^2 $ as $s \uparrow 1$.
    Therefore since $ \lim _{s \uparrow 1} \hbar (s) = 1$ for GF $\mu (s)$ we have
\begin{equation}
    \mu (s) \sim {4 \over {(\alpha  - 1)^2 }}{1 \over {\left( {1 - s}\right)^2 }} \raise 1.5pt\hbox{,}
    \quad \parbox{1.6cm}{\textit{as} {} $ s \uparrow 1 $.}
\end{equation}
    According to Hardy-Littlewood Tauberian theorem each of relations (5.9) and (5.10) entails another.
\end{proof}

    Another invariant measure for Q-process are numbers
\begin{equation}
    \upsilon _j := \mathop {\lim }\limits_{n \to \infty } {{{\cal Q}_{ij} (n)}
    \over {{\cal Q}_{i1} (n)}}  \raise 1.5pt\hbox{,}
\end{equation}
    which don't depend on $i \in {\cal E}$. In fact a similar way as in GWP (see Lemma 3) case it is
    easy to see that this limit exists. Owing to Kolmogorov-Chapman equation
$$
    {{{\cal Q}_{ij} (n + 1)} \over {{\cal Q}_{i1} (n + 1)}}{{{\cal Q}_{i1} (n + 1)} \over {{\cal Q}_{i1} (n)}}
    = \sum\limits_{k \in {\cal E}} {{{{\cal Q}_{ik} (n)} \over {{\cal Q}_{i1} (n)}}{\cal Q}_{kj} (1)}.
$$
    Last equality and (5.11), taking into account that $ {{{\cal Q}_{i1} (n + 1)} \mathord{\left/
    {\vphantom {{{\cal Q}_{i1} (n + 1)} {{\cal Q}_{i1} (n)}}} \right. \kern-\nulldelimiterspace}
    {{\cal Q}_{i1} (n)}} \to 1$ gives us an invariant relation
\begin{equation}
    \upsilon _j  = \sum\nolimits_{i \in {\cal E}} {\upsilon _i {\cal Q}_{ij} (1)}.
\end{equation}
    In GF context the equality (5.12) is equivalent to Schroeder type functional equation
$$
    {\cal U}\left( {\widehat F(s)} \right) = {{\widehat F(s)} \over {Y(s)}}{\cal U}(s),
$$
    where $ \widehat{F_n }(s) = {{F_n (qs)} \mathord{\left/
    {\vphantom {{F_n (qs)} q}} \right. \kern-\nulldelimiterspace} q}$ and
$$
    {\cal U}(s) = \sum\nolimits_{j \in {\cal E}} {\upsilon _j s^j }
$$
    with $\upsilon _1  = 1$.

    Note that in conditions of Theorem 11
$$
    {\cal U}(s) = \pi (s)e^{{{2\gamma } \mathord{\left/ {\vphantom {{2\gamma }
    {(2 + \gamma )}}} \right. \kern-\nulldelimiterspace} {(2 + \gamma )}}}.
$$
    Hence, considering (5.11), we generalize the statement (5.7):
$$
    {\cal Q}_{ij} (n) \longrightarrow \pi _j  = \upsilon _j e^{ - {{2\gamma }\mathord{\left/
    {\vphantom {{2\gamma }{(2 + \gamma )}}}\right. \kern-\nulldelimiterspace}{(2 + \gamma )}}},
    \quad \parbox{2cm}{\textit{as} {} $ n \to \infty $,}
$$
    for all $i,j \in {\cal E}$.

    By similar way for $\beta  = 1$ it is discovered that
$$
    n^2 {\cal Q}_{ij} (n) \longrightarrow \mu _j  = \upsilon _j {{2\widetilde{\cal Q}_1 } \over
    {(\alpha - 1)p_0}} \raise 1.5pt\hbox{,} \quad \parbox{2cm}{\textit{as} {} $ n \to \infty $,}
$$
    where $\widetilde{\cal Q}_1 $ is defined in (5.8).

    Providing that $Y''(1) < \infty $ it can be estimated the convergence speed in Theorem 12.
    It is proved in  {\cite{NMuh}} that if $C: = F'''(1) < \infty $, then
\begin{equation}
    R_n (s) = {1 \over {b_n (s)}} + \Delta  \cdot {{\ln b_n (s) + K(s)}
    \over {\bigl( {b_n (s)} \bigr)^2 }}\bigl( {1 + o(1)} \bigr),
\end{equation}
    as $n \to \infty $, where
$$
    b_n (s) = {{F''(1)} \over 2}n + {1\over {1 - s}}
        \qquad \mbox{\textit{and}} \qquad
    \Delta  = {C \over {3F''(1)}} - {{F''(1)} \over 2} \raise 1.5pt\hbox{,}
$$
    and $K(s)$ is some bounded function depending on form of $F(s)$.
    Since the finiteness of $C$ is equivalent to condition $Y''(1) < \infty $ then
    from combination of relations (2.12), (4.2), (4.3) and (5.13) we receive
    the following theorem for the case $\beta  = 1$.

\begin{theorem}
    If together with conditions of Theorem 12 we suppose that $Y''(1)< \infty $, then for
    the error term in asymptotic formula (5.5) the following estimation holds:
$$
    r_n (s) = \widetilde\Delta \cdot {{\ln b_n (s)} \over {b_n(s)}}\left( {1 + o(1)}
    \right), \quad \parbox{2cm}{\textit{as} {} $ n \to \infty $,}
$$
    where $\widetilde\Delta $ is constant depending on the moment $Y''(1)$ and
$$
    b_n (s) = {{(\alpha  - 1)n} \over 2} + {1 \over {1 - s}} \raise 1.5pt\hbox{.}
$$
\end{theorem}

\begin{corollary}
    In conditions of Theorem 14 the following representation holds:
$$
    n^2 {\cal Q}_{ij} (n) = \mu _j \left( {1 + {\Delta  \over {\alpha - 1}} \cdot {{\ln n}
    \over n}\left( {1 + o(1)} \right)}\right), \quad \parbox{2cm}{\textit{as} {} $ n \to \infty $.}
$$
\end{corollary}

\medskip

\section{ Joint distribution law of Q-process and its total state}

    Consider the Q-process $\left\{ {W_n , n \in {\mathbb{N}}_0 }\right\}$ with
    structural parameter $\beta  = F'(q)$. Let's define a random variable
$$
    S_n  = W_0  + W_1  + \, \cdots \, + W_{n - 1},
$$
    a total state in Q-process until time $n$. Let
$$
    J_n (s;x) = \sum\limits_{j \in {\cal E}} {\sum\limits_{l \in {\mathbb{N}}}
    {\mathbb{P}\bigl\{ {W_n  = j, S_n  = l} \bigr\}s^j x^l } }
$$
    be the joint GF of $W_n $ and $S_n $ on a set of
$$
    \mathbb{K} = \left\{ {(s;x) \in {\mathbb{R}}^2: \; |s| \le 1,\; |x|
    \le 1, \; \sqrt {(s - 1)^2  + (x - 1)^2 }  \ge r > 0} \right\}.
$$

\begin{lemma}
    For all $(s;x) \in \mathbb{K}$ and any $n \in {\mathbb{N}}$ a recursive equation
\begin{equation}
    J_{n + 1} (s;x) = {{Y(s)} \over {\widehat F(s)}}J_n \left({x\widehat F(s);x} \right)
\end{equation}
    holds, where $Y(s) = s{{F'(qs)} \mathord{\left/ {\vphantom {{F'(qs)} \beta }} \right.
    \kern-\nulldelimiterspace} \beta }$ and $\widehat F(s) = {{F(qs)} \mathord{\left/
    {\vphantom {{F(qs)} q}} \right. \kern-\nulldelimiterspace} q}$.
\end{lemma}

\begin{proof}
    Let's consider the cumulative process $\bigl\{ {W_n , S_n }\bigr\}$ which is evidently a
    bivariate Markov chain with transition functions
$$
    \mathbb{P}\bigl\{{W_{n + 1} = j,\,S_{n + 1} = l\bigm| {W_n = i,\,S_n = k}} \bigr\}
    = \mathbb{P}_i \bigl\{ {W_1  = j,\,S_1 = l} \bigr\}\delta _{l, i + k},
$$
    where $\delta _{ij} $ is the Kronecker's delta function. Hence we have
\begin{eqnarray}
    \mathbb{E}_i \Bigl[ {s^{W_{n + 1} } x^{S_{n + 1} } \bigm| {S_n  = k}} \Bigr] \nonumber
    & = & \sum\limits_{j \in {\cal E}} {\sum\limits_{l \in {\mathbb{N}}} {\mathbb{P}_i \bigl\{ {W_1
    = j,\, S_1  = l} \bigr\}\delta _{l,i + k} s^j x^l } } \\
    & = & \sum\limits_{j \in {\cal E}} {\mathbb{P}_i \bigl\{ {W_1 = j} \bigr\}s^j x^{i + k}}
    = Y^{(i)} (s) \cdot x^{i + k}. \nonumber
\end{eqnarray}
    Using this result and the formula of composite probabilities, we discover that
\begin{eqnarray}
    J_{n + 1} (s;x) \nonumber
    & = & \mathbb{E}\Bigr[ {\mathbb{E}\bigl[ {s^{W_{n + 1}} x^{S_{n + 1}} \bigm| {W_n , S_n }} \bigr]} \Bigr]
    = \mathbb{E}\left[ {Y^{(W_n )} (s) \cdot x^{W_n  + S_n } } \right] \\  \nonumber
    & = & \mathbb{E}\left[ {\left({\widehat F(s)} \right)^{W_n -1} \cdot Y(s) \cdot x^{W_n + S_n }} \right] \\ \nonumber
    & = & {{Y(s)}\over {\widehat F(s)}}\cdot \mathbb{E}\left[{\left({x\widehat F(s)}\right)^{W_n} \cdot x^{S_n }} \right].
\end{eqnarray}
    The formula (4.2.) is used in last step. The last equation reduces to (6.1).
\end{proof}

    Now by means of relation (6.1) we can take an explicit expression for
    GF $ J_n (s;x)$. In fact, sequentially having applied it, taking into account(4.4) and,
    after some transformations we have
\begin{equation}
    J_n (s;x) = s\prod\limits_{k = 0}^{n - 1} {\left[ {{{x\widehat F^\prime
    \left( {H_k (s;x)} \right)} \over \beta }} \right]}
    = {s \over {\beta ^n }}{{\partial H_n (s;x)} \over {\partial s}}  \raise 1.2pt\hbox{,}
\end{equation}
    where the sequence of functions $\left\{ {H_k (s;x)}\right\}$ is defined
    for $(s;x) \in \mathbb{K}$ by following recurrence relations:
\begin{eqnarray}
   H_0 (s;x) & = & s,   \nonumber\\
   H_{n + 1}(s;x) & = &  x\widehat F\bigl( {H_n (s;x)} \bigr).
\end{eqnarray}

    Since
$$
    \left. {{{\partial J_n (s;x)} \over {\partial x}}} \right|_{(s;x)= (1;1)} = \mathbb{E}S_n,
$$
    then provided that $\alpha : = Y'(1)$ it follows from 6.2) and (6.3) that
\begin{equation}
    \mathbb{E}S_n  =\left\{\begin{array}{l} (1 + \gamma )n - \gamma {\displaystyle{1 - \beta ^n }
    \over \displaystyle{1 - \beta }}  \hfill \raise 1pt\hbox{,}
    \qquad \parbox{2.2cm}{\textit{when} {} $\beta < 1 $,}  \\
    \\
    {\displaystyle{\alpha  - 1} \over \displaystyle 2}n(n - 1) + n \hfill,
    \qquad  \parbox{2.2cm}{\textit{when} {} $\beta = 1 $,}  \\
    \end{array} \right.
\end{equation}
    where as before $\gamma : = {{\left( {\alpha  - 1} \right)}\mathord{\left/
    {\vphantom {{\left( {\alpha - 1} \right)}{\left({1 - \beta} \right)}}} \right.
    \kern-\nulldelimiterspace} {\left( {1 - \beta } \right)}}$.

\begin{remark}
    It is known from classical theory that if an evolution law of simple GWP
    $\left\{ {\widehat{Z_n }, n \in {\mathbb{N}}_0 } \right\}$ is generated by GF
    $\widehat F(s) = {{F(qs)} \mathord{\left/ {\vphantom {{F(qs)} q}} \right. \kern-\nulldelimiterspace} q}$,
    then a joint GF of distribution of $\left\{ {\widehat{Z_n }, V_n} \right\}$, where
    $V_n  = \sum\nolimits_{k = 0}^{n - 1}{\widehat{Z_k }} $ is the total number of individuals
    participating until time $n$, satisfies to the recurrent equation (6.3); see e.g., {\cite[p.126]{Kolchin}}.
    So $H_n (s;x)$, $(s;x) \in \mathbb{K}$, represents the two-dimensional GF for
    all $n \in {\mathbb{N}}$ and has all properties as $\mathbb{E}\left[ {s^{\widehat Z_n } x^{V_n }} \right]$.
\end{remark}

    In virtue of the told in Remark 6, in studying of function $H_k(s;x)$ we certainly will use
    properties of GF $\mathbb{E}\left[ {s^{\widehat Z_n } x^{V_n } } \right]$.
    As well as $\widehat F^\prime (1) = \beta  \le 1$ and hence the process
    $\left\{ {\widehat{Z_n }, n \in {\mathbb{N}}_0 } \right\}$ is mortal GWP.
    So there is an integer valued random variable $V = \lim _{n \to \infty } V_n $ -- a total number
    of individuals participating in the process for all time of its evolution.
    Hence there is a limit
$$
    h(x): = \mathbb{E}x^V  = \lim _{n \to \infty } \mathbb{E}x^{V_n } = \lim _{n \to \infty } H_n (1;x)
$$
    and according to (6.3) it satisfied the recurrence relation
\begin{equation}
    h(x) = x\widehat F\bigl( {h(x)} \bigr).
\end{equation}

    Provided that the second moment $Y''(1)$ is finite, the following asymptotes for
    the variances can be found from (6.2) by differentiation:
\begin{equation}
    \textsf{Var} W_n  \sim \left\{ \begin{array}{l} \mathcal{O}(1)  \hfill ,
    \qquad \parbox{2.2cm}{\textit{when} {} $\beta < 1 $,} \\ \nonumber
    \\
    {\displaystyle{\left( {\alpha  - 1} \right)^2} \over \displaystyle 2}n^2 \, \hfill,
    \qquad  \parbox{2.2cm}{\textit{when} {} $\beta = 1 $,}  \\
    \end{array} \right.
\end{equation}
    and
\begin{equation}
    \textsf{Var} S_n  \sim \left\{ \begin{array}{l} \mathcal{O}(n)  \hfill ,
    \qquad \parbox{2.2cm}{\textit{when} {} $\beta < 1 $,} \\ \nonumber
    \\
    {\displaystyle{\left( {\alpha  - 1} \right)^2} \over \displaystyle 12}n^4 \, \hfill,
    \qquad  \parbox{2.2cm}{\textit{when} {} $\beta = 1 $,}  \\
    \end{array} \right.
\end{equation}
    as $n \to \infty $. In turn it is matter of computation to verify that
\begin{equation}
    \textsf{cov} \bigl( {W_n , S_n } \bigr)  \sim \left\{ \begin{array}{l} \mathcal{O}(1)  \hfill ,
    \qquad \parbox{2.2cm}{\textit{when} {} $\beta < 1 $,} \\ \nonumber
    \\
    {\displaystyle{\left( {\alpha  - 1} \right)^2} \over \displaystyle 6}n^3 \, \hfill,
    \qquad  \parbox{2.2cm}{\textit{when} {} $\beta = 1 $.}  \\
    \end{array} \right.
\end{equation}
    Hence letting $\rho _n $ denote the correlation coefficient of $W_n $ and $S_n $, we have
\begin{equation}
    \mathop {\lim }\limits_{n \to \infty } \rho _n  = \left\{ \begin{array}{l} 0  \hfill ,
    \qquad \parbox{2.2cm}{\textit{when} {} $\beta < 1 $,} \\ \nonumber
    \\
    {\displaystyle{\sqrt 6} \over \displaystyle 3} \, \hfill \hfill  \raise 1pt\hbox{,}
    \qquad  \parbox{2.2cm}{\textit{when} {} $\beta = 1 $.}  \\
    \end{array} \right.
\end{equation}

    Last statement specifies that in the case $\beta  < 1$ between the variables $W_n $ and $S_n $
    there is an asymptotic independence property. Contrariwise for the case $\beta  = 1$ the
    following "joint theorem" \,holds, which has been proved in the paper {\cite{Imomov14b}}.

\begin{theorem}
    Let $\beta  = 1$ and $\alpha  = Y'(1) < \infty $. Then the two-dimensional process
$$
    \left( {{{W_n } \over {{\mathbb{E}}W_n}};\, {{S_n } \over {{\mathbb{E}}S_n }}}\right)
$$
    weakly converges to the two-dimensional random vector
    $\left( {\textbf{\textsf{w}};\textbf{\textsf{s}}}\right)$ having the Laplace transform
$$
    {\mathbb{E}}\left[ {e^{ - \lambda \textbf{\textsf{w}} - \theta \textbf{\textsf{s}}}} \right]
    = \left[{{\rm{ch}}\sqrt \theta + {\lambda  \over 2}{{{\rm{sh}}\sqrt \theta }
    \over {\sqrt \theta  }}} \right]^{ - 2} , \;\; \lambda , \theta \in {\mathbb{R}}_+ ,
$$
    where ${\rm{ch}} x = {{\bigl( {e^x  + e^{ - x} } \bigr)} \mathord{\left/
    {\vphantom {{\bigl( {e^x  + e^{ - x} } \bigr)} 2}} \right. \kern-\nulldelimiterspace} 2}$ and
    $ {\rm{sh}} x = {{\bigl( {e^x  - e^{ - x} } \bigr)} \mathord{\left/
    {\vphantom {{\bigl( {e^x  - e^{ - x} } \bigr)} 2}} \right. \kern-\nulldelimiterspace} 2}$.
\end{theorem}

    Supposing $\lambda  = 0$ in Theorem 15
    produces the following limit theorem for $S_n $.

\begin{corollary}
    Let $\beta  = 1$ and $\alpha  = Y'(1) < \infty $. Then for $0 < u < \infty $
$$
    \mathop {\lim }\limits_{n \to \infty } {\mathbb{P}}\left\{ {{{S_n }
    \over {{\mathbb{E}}S_n }} \le u} \right\} = F(u),
$$
    where the limit function $F(u)$ has the Laplace transform
$$\
    \int_0^{ + \infty } {e^{ - \theta u} dF(u)}
    = {\rm{sech}}^2 \sqrt \theta  \,,\;\; \theta  \in {\mathbb{R}}_ + .
$$
\end{corollary}

    Letting $\theta  = 0$ from the Theorem 15 we have the following assertion which
    was proved in the monograph {\cite[pp.59--60]{ANey}} with applying of the Helly's theorem.

\begin{corollary}
    Let $\beta  = 1$ and $\alpha  = Y'(1) < \infty $. Then for $0 < u < \infty $
\begin{equation}
    \mathop {\lim }\limits_{n \to \infty } {\mathbb{P}}\left\{ {{{W_n }
    \over {{\mathbb{E}}W_n }} \le u} \right\} = 1 - e^{ - 2u} - 2ue^{-2u}.
\end{equation}
\end{corollary}

\begin*{\textit{Really}}, denoting $\psi _n (\lambda ) = \Psi _n (\lambda ;0)$ we have
$$
    \psi _n (\lambda ) \longrightarrow {1 \over {\left[ {1 + {\displaystyle \lambda
    \over \displaystyle 2}}\right]^2 }}  \raise 1pt\hbox{,}
    \quad \parbox{2cm}{\textit{as} {} $ n \to \infty $.}
$$
    Here we have used that $\lim _{\theta  \downarrow 0} {{{\rm{sh}}\sqrt \theta  } \mathord{\left/
    {\vphantom {{{\rm{sh}}\sqrt \theta  } {\sqrt \theta  }}} \right.
    \kern-\nulldelimiterspace}{\sqrt \theta }} = 1$. The found Laplace transform corresponds
    to a distribution of the right-hand side term in (6.6) produced as composition of
    two exponential laws with an identical density.
\end*{}

\medskip

\section{ Asymptotic properties of $S_n $ in case of $\beta < 1$}

    In this section we investigate asymptotic properties of distribution of
    $S_n $ in the case $\beta  < 1$. Consider the
    GF $T_n (x): = \mathbb{E}x^{S_n } = J_n (1;x)$. Owing to (6.2) it has a form of
\begin{equation}
    T_n (x) = \prod\limits_{k = 0}^{n - 1} {u_k (x)},
\end{equation}
    where
$$
    u_n (x) = {{x\widehat F^\prime  \left( {h_n (x)} \right)} \over \beta } \raise 1pt\hbox{,}
$$
    and $ \widehat F(s) = {{F(qs)} \mathord{\left/ {\vphantom {{F(qs)} q}} \right. \kern-\nulldelimiterspace} q}$,
    $h_n (x) = \mathbb{E}x^{V_n } $, $V_n  = \sum \nolimits_{k = 0}^{n - 1} {\widehat{Z_k }} $.

    In accordance with (6.3) $h_{n + 1} (x) = x\widehat F\bigl( {h_n(x)} \bigr)$.
    Denoting
$$
    R_n (x): = h(x) - h_n (x), \;n \in {\mathbb{N}}_0,
$$
    for $x \in \mathbb{K}$ we have
\begin{eqnarray}
    R_n (x) \nonumber
    & = & x\left[ {\widehat F\left( {h(x)} \right) - \widehat F\left( {h_{n - 1} (x)} \right)} \right] \\
    & = & x\mathbb{E}\bigl[ {h(x) - h_{n - 1} (x)} \bigr]^{\widehat Z_n }  \le \beta R_{n - 1} (x), \nonumber
\end{eqnarray}
    since $\left| {h(x)} \right| \le 1$ and $\left| {h_n (s;x)} \right| \le 1$. Therefore
$$
    \bigl| {R_n (x)} \bigr| \le \beta ^{n - k} \bigl| {R_k (x)}\bigr|,
$$
    for each $n \in {\mathbb{N}}$ and $k = 0,1, \, \ldots \,, n$.
    Consecutive application of last inequality gives
\begin{equation}
    R_n (x) = \mathcal{O}\left( {\beta ^n } \right) \longrightarrow 0,
\end{equation}
    as $n \to \infty $ uniformly for $x \in \mathbb{K}$. Further, where the
    function $R_n (x)$ is used, we deal with  set $\mathbb{K}$ in
    which this function certainly is not zero.

    By Taylor expansion and taking into account (7.2), (6.5), we have
\begin{equation}
    R_{n + 1} (x) = x\widehat F^\prime \bigl( {h(x)} \bigr)R_n (x)
    - x{{\widehat F^{\prime \prime} \bigl( {h(x)} \bigr)
    + \eta _n (x)}\over 2}R_n^2 (x),
\end{equation}
    where $\left| {\eta _n (x)} \right| \to 0$ as $n \to \infty$ uniformly
    with respect to $x \in \mathbb{K}$.  Since $R_n (x) \to 0$, formula (7.3) implies
$$
    R_n (x) = {{R_{n + 1}(x)} \over {x\widehat F^\prime \bigl({h(x)} \bigr)}}
    \bigl( {1 + o(1)} \bigr).
$$
    Owing to last equality we transform the formula (7.3) to a form of
$$
    R_{n + 1} (x) = x\widehat F^\prime  \bigl( {h(x)} \bigr)R_n (x)
    - \left[ {{{\widehat F^{ \prime \prime } \bigl( {h(x)} \bigr)}
    \over {2\widehat F^\prime  \bigl( {h(x)} \bigr)}} + \varepsilon_n (x)}
    \right]R_n (x)R_{n + 1} (x)
$$
    and, hence
\begin{equation}
    {{u(x)} \over {R_{n + 1} (x)}} = {1 \over {R_n (x)}} + v(x) + \varepsilon _n (x),
\end{equation}
    where
$$
    u(x) = x\widehat F^\prime  \bigl( {h(x)} \bigr)
        \qquad \mbox{\textit{and}} \qquad
    v(x) = {{\widehat F^ {\prime  \prime } \bigl( {h(x)} \bigr)}
    \over {2\widehat F^\prime  \bigl( {h(x)} \bigr)}}  \raise 1pt\hbox{,}
$$
    and $ \left| {\varepsilon _n (x)} \right| \le \varepsilon _n  \to 0$
    as $n \to \infty $  for all $x \in \mathbb{K}$. Repeated use of (7.4)
    leads to the following representation for $R_n (x)$:
\begin{equation}
    {{u^n (x)} \over {R_n (x)}} = {1 \over {h(x) - 1}} +{{v(x)\cdot \bigl[ {1 - u^n (x)} \bigr]}
    \over {1 - u(x)}} + \sum\limits_{k = 1}^n {\varepsilon _k (x)u^k (x)}.
\end{equation}

    Note that the formula (7.5) was written out in monograph {\cite[p.130]{Kolchin}}  for the critical case.

    The expansions of functions $h(x)$ and $u(x)$ in neighborhood
    of $x = 1$ will be useful for our further purpose.

\begin{lemma}
    Let $\beta  < 1$. If $b: = \widehat F^ {\prime  \prime }(1) < \infty$,
    then for $h(x) = \mathbb{E}x^V $ the following relation holds:
\begin{equation}
    1 - h(x) \sim {1 \over {1 - \beta }}\,(1 - x)
    - {{2\beta (1 -\beta ) + b} \over {(1 - \beta )^3 }}\,(1 - x)^2,
\end{equation}
    as $x \uparrow 1$.
\end{lemma}

\begin{proof}
    We write down the Taylor expansion as $x \uparrow 1$:
\begin{equation}
    h(x) = 1 + h'(1)\bigl(x - 1\bigr)
    + h''(1)\bigl(x - 1\bigr)^2  + o\bigl(x - 1\bigr)^2.
\end{equation}
    In turn by direct differentiation from (6.5) we have
$$
    h'(x) = {{\widehat F\bigl( {h(x)} \bigr)} \over {1 - u(x)}} \raise 1pt\hbox{,}
$$
    and
$$
    h''(x) = {{2\widehat F^\prime \bigl( {h(x)} \bigr)h'(x) + x\widehat F^{\prime \prime}
    \bigl( {h(x)} \bigr)\bigl[ {h'(x)}\bigr]^2 } \over {1 - u(x)}}  \raise 1pt\hbox{.}
$$
    Letting $x \uparrow 1$ in last equalities entails $h'(1) = {1 \mathord{\left/
    {\vphantom {1 {(1 - \beta )}}} \right. \kern-\nulldelimiterspace} {(1 - \beta )}}$ and
$$
    h''(1) ={ {2\beta (1 - \beta ) + b} \over {(1 - \beta )^3 }}
$$
    which together with (7.7) proves (7.6).
\end{proof}

    We remind that existence of the second moment $b: = \widehat F^{\prime \prime}(1)$ is
    equivalent to existence of $\alpha  = Y'(1)$ and $\gamma  = {b \mathord{\left/
    {\vphantom {b {\beta (1 - \beta )}}} \right. \kern-\nulldelimiterspace} {\beta (1 - \beta )}}$.
    We use it in the following assertion.
\begin{lemma}
    Let $\beta  < 1$. If $b: = \widehat F^ {\prime  \prime }(1) < \infty$,
    then as $x \uparrow 1$ the following relation holds:
\begin{equation}
    u(x) \sim \beta x\left[ {1 - \gamma \,(1 - x)} \right]
    + {{2\beta (1 - \beta ) + b} \over {(1 - \beta)^3 }}bx\,(1 - x)^2.
\end{equation}
\end{lemma}

\begin{proof} The relation (7.8) follows from Taylor power series expansion of
    function $\widehat F^\prime  \left( {h(x)} \right)$, taking into account therein Lemma 5.
\end{proof}

    The following Lemma 7 is a direct consequence of relation (7.6).
    And Lemma 8 implies from (7.8) and Lemma 7.
    Therein we consider the fact that $b = \beta(\alpha  - 1)$.

\begin{lemma}
    Let $\beta  < 1$ and $\alpha  < \infty $. Then as $\theta  \to 0$
\begin{equation}
    h\left({e^\theta  } \right) - 1 \sim {1 \over {1 - \beta }}\theta
    + {{\beta (2 + \gamma )} \over {(1 - \beta )^2 }}\,\theta ^2.
\end{equation}
\end{lemma}

\begin{lemma}
    If $\beta  < 1$ and $\alpha  < \infty $, then as $\theta  \to 0$
\begin{equation}
    u\left( {e^\theta  } \right) \sim \beta \left[ {1 + (1 + \gamma)\theta } \right]
    + \beta \gamma {{1 + \beta (1 + \gamma )} \over {1 - \beta }}\,\theta ^2.
\end{equation}
\end{lemma}

    The following assertion hails from (7.5), (7.9) and (7.10).

\begin{lemma}
    Let $\beta  < 1$ and $\alpha  < \infty $. Then the following relation holds:
\begin{equation}
    {{R_n \left( {e^\theta  } \right)} \over {u^n \left( {e^\theta  }\right)}}
    \sim {1 \over {1 - \beta }}\theta  + {{\beta (2 + \gamma)} \over {(1 - \beta )^2 }}\,\theta ^2,
\end{equation}
   as $\theta  \to 0$ and for each fixed $n \in {\mathbb{N}}$.
\end{lemma}

    Further the following lemma is required.

\begin{lemma}
    Let $\beta  < 1$ and $\alpha  < \infty $. Then the following relation holds:
\begin{equation}
    \ln \prod\limits_{k = 0}^{n - 1} {u_k \left( {e^\theta  } \right)} \sim
    - \left( {1 - {{u\left( {e^\theta  } \right)} \over \beta }} \right)n
    - {{\beta \gamma (2 + \gamma )} \over {1 - \beta}}\,\theta ^3
    \sum\limits_{k = 0}^{n - 1} {u^k \left( {e^\theta  }\right)},
\end{equation}
   as $\theta  \to 0$ and for each fixed $n \in {\mathbb{N}}$.
\end{lemma}

\begin{proof}
    Using inequalities $\ln (1 - y) \ge  - y - {{y^2 } \mathord{\left/ {\vphantom {{y^2 }
    {\left( {1 - y} \right)}}} \right. \kern-\nulldelimiterspace} {\left( {1 - y} \right)}}$,
    which hold for $0 \le y < 1$, we have
\begin{eqnarray}
    \ln \prod\limits_{k = 0}^{n - 1} {u_k \left( {e^\theta  } \right)} \nonumber
    & = & \sum\limits_{k = 0}^{n - 1} {\ln \left\{ {1 - \left[ {1 - u_k \left( {e^\theta  } \right)} \right]} \right\}} \\
    & = & \sum\limits_{k = 0}^{n - 1} {\left[ { u_k \left( {e^\theta  } \right)}  - 1 \right]}  + \rho _n^{(1)} (\theta )
    = :I_n (\theta ) + \rho _n^{(1)} (\theta ),
\end{eqnarray}
    where
\begin{equation}
    I_n (\theta ) =  - \sum\limits_{k = 0}^{n - 1} {\left[ {1 - u_k\left( {e^\theta  } \right)} \right]},
\end{equation}
    and
$$
    0 \ge \rho _n^{(1)} (\theta ) \ge - \sum\limits_{k = 0}^{n - 1} {{{\left[ {1 - u_k \left( {e^\theta }
    \right)} \right]^2 } \over{u_k \left( {e^\theta  } \right)}}}  \raise 1pt\hbox{.}
$$

    It is easy to be convinced that the functional sequence $\left\{{h_k (x)} \right\}$ does not
    decrease on $k$. Then according to property of GF, the function $u_k \left({e^\theta  } \right)$ is
    also non-decreasing on $k$ for each fixed $n \in {\mathbb{N}}$ and $\theta  \in {\mathbb{R}}$. Hence,
\begin{equation}
    0 \ge \rho _n^{(1)} (\theta ) \ge {{1 - u_0 \left( {e^\theta  }\right)}
    \over {u_0 \left( {e^\theta } \right)}}I_n (\theta ).
\end{equation}
    We can verify also that $1 -u_0 \left( {e^\theta } \right) \to 0$ as $\theta  \to 0$.
    Then in accordance with (7.15) the second expression in (7.13) $\rho _n^{(1)} (\theta ) \to 0$
    provided that $I_n (\theta )$ has a finite limit as $\theta  \to 0$.

    Further, by Taylor expansion we have
$$
    \widehat F^\prime (t) = \widehat F^\prime  (t_0 )
    - \widehat F^{\prime  \prime} (t_0 )(t_0 - t) + (t_0  - t)g(t_0 ;t),
$$
    where $g(t_0 ;t) = (t_0  - t){{\widehat F^{\prime  \prime  \prime} (\tau)}
    \mathord{\left/{\vphantom {{\widehat F^{\prime \prime \prime}(\tau )}2}} \right.
    \kern-\nulldelimiterspace} 2}$ and $t_0  < \tau  < t$. Using this expansion we write
$$
    u_k (x) = {{u(x)} \over \beta } - {{x\widehat F^{\prime  \prime}
    \bigl( {h(x)} \bigr)} \over \beta }R_k (x) + R_k (x)g_k (x),
$$
    herein $g_k (x) = xR_k (x){{\widehat F^{\prime \prime \prime} }\mathord{\left/
    {\vphantom {{\widehat F^{\prime \prime \prime}}{2\beta }}} \right.
    \kern-\nulldelimiterspace} {2\beta }}$ and $h_k (x) < \tau  < h(x)$.
    Therefore
\begin{equation}
    u_k \left( {e^\theta  } \right) = {{u\left( {e^\theta  } \right)}\over \beta }
    - {{e^\theta  \widehat F^{\prime  \prime}  \left({h\left( {e^\theta  } \right)}
    \right)} \over \beta }R_k \left({e^\theta  } \right) + R_k \left( {e^\theta  }
    \right)g_k \left({e^\theta  } \right).
\end{equation}
    It follows from (7.14) and (7.16) that
\begin{equation}
    I_n (\theta ) =  - \left[ {1 - {{u\left( {e^\theta  } \right)}\over \beta }}
    \right]n - {{e^\theta  \widehat F^{\prime  \prime}\left( {h\left( {e^\theta }
    \right)} \right)} \over \beta}\sum\limits_{k = 0}^{n - 1} {R_k \left( {e^\theta } \right)}
    + \rho _n^{(2)} (\theta ),
\end{equation}
    where
$$
    0 \le \rho _n^{(2)} (\theta ) \le R_0\left( {e^\theta  }\right)\sum\limits_{k = 0}^{n - 1}
    {g_k \left({e^\theta  } \right)}.
$$
    In last estimation we used the earlier known inequality $\left|{R_n (x)} \right| \le \beta ^n
    \left| {R_0 (x)} \right|$. Owingto the relation (7.9) $R_0 \left( {e^\theta } \right)
    = \mathcal{O}(\theta)$ as $\theta  \to 0$. In turn according to (7.2) $g_k \left( {e^\theta  } \right)
    =\mathcal{O}\left( {\beta ^k } \right) \to 0$ as $k \to \infty $ for all $\theta  \in {\mathbb{R}}$.
    Hence,
$$
    R_0 \left( {e^\theta  } \right)\sum\limits_{k = 0}^{n - 1} {g_k\left( {e^\theta  }
    \right)}  = \mathcal{O}(\theta ) \longrightarrow 0, \quad \parbox{2cm}{\textit{as} {} $ \theta \to 0 $.}
$$
    It follows from here that the error term in (7.17)
\begin{equation}
    \rho _n^{(2)} (\theta ) \longrightarrow 0, \quad \parbox{2cm}{\textit{as} {} $ \theta \to 0 $.}
\end{equation}
    Considering together (7.11), (7.17) and (7.18) and, after some computation,
    taking into account a continuity property of $\widehat F^{\prime  \prime }(s)$, we obtain (7.12).

    The Lemma is proved.
\end{proof}

    With the help of the above established lemmas, we state and prove now
    the analogue of Law of Large Numbers and the Central Limit Theorem for $S_n $.

\begin{theorem}
    Let $\beta  < 1$ and $\alpha  < \infty $. Then
\begin{equation}
    \mathop {\lim }\limits_{n \to \infty } \mathbb{P}\left\{ {{{S_n } \over n}< u} \right\}
    = \left\{ \begin{array}{l} 0  \hfill ,
    \qquad \parbox{2.3cm}{\textit{if} {} $ u < 1 + \gamma $,} \\ \nonumber
    \\
    1  \hfill ,   \qquad  \parbox{2.3cm}{\textit{if} {} $ u \ge 1 + \gamma $,} \\
    \end{array} \right.
\end{equation}
    where $\gamma  = {{(\alpha  - 1)} \mathord{\left/ {\vphantom {{(\alpha  - 1)}
    {(1 - \beta )}}} \right. \kern-\nulldelimiterspace} {(1 - \beta )}}$.
\end{theorem}

\begin{proof}
    Denoting $\psi _n (\theta )$ be the Laplace transform of  distribution
    of ${{S_n }\mathord{\left/ {\vphantom {{S_n } n}} \right. \kern-\nulldelimiterspace} n}$ it follows
    from formula (7.1) that $\psi _n (\theta ) = T_n \left( {\theta _n } \right)$, where
    $\theta _n  = \exp \left\{ { - {\theta  \mathord{\left/ {\vphantom {\theta  n}} \right.
    \kern-\nulldelimiterspace} n}} \right\}$. The theorem statement is equivalent to that for
    any fixed $\theta \in {\mathbb{R}}_ +  $
\begin{equation}
    \psi _n (\theta ) \longrightarrow e^{ - \theta (1 + \gamma )},
    \quad \parbox{2cm}{\textit{as} {} $ n \to\infty $.}
\end{equation}
    From Lemma 10 follows
\begin{equation}
    \ln \psi _n (\theta ) \sim  - \left( {1 - {{u\left( {\theta _n }\right)} \over \beta }}
    \right)n + {{\beta \gamma (2 + \gamma )}\over {1 - \beta }}\,{{\theta ^3 } \over {n^3 }}
    \sum\limits_{k =0}^{n - 1} {u^k \left( {\theta _n } \right)},
\end{equation}
    as $n \to \infty $. The first addendum, owing to (7.10), becomes
\begin{equation}
    \left( {1 - {{u\left( {\theta _n } \right)} \over \beta }}\right)n \sim (1 + \gamma )\theta
    - \gamma {{1 + \beta (1 +\gamma )} \over {1 - \beta }}\,{{\, \theta ^2 } \over n}\raise 1pt\hbox{.}
\end{equation}
    And the second one, as it is easy to see, has a decrease order of
    $\mathcal{O}\left( {{1 \mathord{\left/ {\vphantom {1 {n^3 }}} \right.
    \kern-\nulldelimiterspace} {n^3 }}} \right)$.
    Therefore from (7.20) and (7.21) follows (7.19).

    The Theorem is proved.
\end{proof}

    We note that in view of the relation (7.21), it can be estimated the
    rate of convergence of $ {{S_n } \mathord{\left/ {\vphantom {{S_n } n}} \right.
    \kern-\nulldelimiterspace} n} \longrightarrow (1 + \gamma )$ as $n \to \infty $.

\begin{theorem}
    Let $\beta  < 1$, $\alpha  < \infty $, and $\gamma = {{(\alpha -1)} \mathord{\left/
    {\vphantom {{(\alpha  - 1)} {(1 - \beta )}}} \right. \kern-\nulldelimiterspace} {(1 - \beta )}}$.
    Then
$$
    \mathbb{P}\left\{ {{{S_n  - \mathbb{E}S_n } \over {\sqrt {2\Psi n} }} < x}
    \right\} \longrightarrow \Phi (x), \quad \parbox{2cm}{\textit{as} {} $ n \to\infty $,}
$$
    where the constant
$$
    \Psi  = \gamma {{1 + \beta (1 + \gamma )} \over {1 - \beta }}
$$
    and $\Phi (x)$ -- the standard normal distribution function.
\end{theorem}

\begin{proof}
    This time let $\varphi _n (\theta )$ be the characteristic function
    of distribution of ${{\bigl( {S_n  - \mathbb{E}S_n } \bigr)}
    \mathord{\left/ {\vphantom {{\bigl( {S_n  - \mathbb{E}S_n } \bigr)}
    {\sqrt {2\Psi n} }}} \right.\kern-\nulldelimiterspace}
    {\sqrt {2\Psi n} }}$:
$$
    \varphi _n \left( \theta  \right): = \mathbb{E}\left[ {\exp {{i\theta
    \left( {S_n  - \mathbb{E}S_n } \right)} \over {\sqrt {2\Psi n} }}}\right].
$$
    According to (6.4) we have
\begin{equation}
    \ln \varphi _n (\theta ) \sim  - (1 + \gamma ){{i\theta n} \over {\sqrt {2\Psi n} }}
    + \ln T_n \left( {\theta _n } \right), \quad \parbox{2cm}{\textit{as} {} $ n \to\infty $,}
\end{equation}
    where $\theta _n  = \exp \left\{ {{{i\theta }\mathord{\left/ {\vphantom {{i\theta }
    {\sqrt {2\Psi n} }}} \right. \kern-\nulldelimiterspace} {\sqrt {2\Psi n} }}} \right\}$.
    Combining (7.1) and Lemma 10 yields
\begin{equation}
    \ln T_n \left( {\theta _n } \right) \sim  - \left( {1 - {{u\left({\theta _n } \right)}
    \over \beta }} \right)n + {{\beta \gamma (2+ \gamma )} \over {1 - \beta }}\,{{i\theta ^3 }
    \over {(2\Psi n)^{{3 \mathord{\left/ {\vphantom {3 2}} \right. \kern-\nulldelimiterspace} 2}} }}
    \sum\limits_{k = 0}^{n - 1} {u^k \left( {\theta _n } \right)}.
\end{equation}
    In turn from (7.10) we have
\begin{equation}
    1 - {{u\left( {\theta _n } \right)} \over \beta } \sim  - (1 + \gamma ){{i\theta }
    \over {\sqrt {2\Psi n} }} - \,{{\, \theta ^2 }\over {\,2n}} \raise 1pt\hbox{.}
\end{equation}
    Using relations (7.23) and (7.24) in (7.22) follows
$$
    \ln \varphi _n (\theta ) =  - {{\theta ^2 } \over 2} + \mathcal{O}\left({{{\theta ^3 } \over
    {n^{{3 \mathord{\left/ {\vphantom {3 2}} \right. \kern-\nulldelimiterspace} 2}}}}}
    \right), \quad \parbox{2cm}{\textit{as} {} $ n \to\infty $.}
$$
    Hence we conclude that
$$
    \varphi _n (\theta ) \longrightarrow \exp \left\{ { - {{\theta ^2 }\over 2}} \right\},
    \quad \parbox{2cm}{\textit{as} {} $ n \to\infty $,}
$$
    and the theorem statement follows from the continuity theorem for characteristic functions.
\end{proof}

\end{document}